\documentclass[final,leqno]{scrarticle}

\usepackage[table]{xcolor}
\usepackage{mathtools}
\usepackage{todonotes}

\usepackage[ruled,vlined]{algorithm2e}
\usepackage{enumerate}
\usepackage{enumitem}
\usepackage{amsmath}
\usepackage{amssymb}
\usepackage{amsthm}
\usepackage{float}
\usepackage{comment}
\usepackage{graphicx}
\usepackage{textcomp}
\usepackage{appendix}
\usepackage{authblk}
\usepackage{bm}
\usepackage{dutchcal}
\usepackage{subcaption}
\usepackage{srcltx} %Hin-und herspringen zwischen Kile und Okular
\usepackage{hyperref}%Darstellung von Web-Adressen
\usepackage{url}
\hypersetup{draft=false, colorlinks=true, linkcolor=black, urlcolor=blue, citecolor=black}
\makeatletter
\g@addto@macro{\UrlBreaks}{\do\/\do-\do\_} % Allow breaks at "/", "-", "_"
\makeatother
\newtheorem{theorem}{Theorem}[section]
\newtheorem{lemma}[theorem]{Lemma}

\newtheorem{algo}[theorem]{Algorithm}
\newtheorem{remark}[theorem]{Remark}

\newtheorem{definition}[theorem]{Definition}

\usepackage{listings}
\newcommand{\dx}{\,\mathrm{d}x} % d in mathrm, x in math mode %Doesn't seem to work
\newcommand{\dt}{\,\mathrm{d}t}
\newcommand{\dH}{\,\mathrm{d}h}
\newcommand{\dsig}{\,\mathrm{d}\sigma}
\newcommand{\dw}{\,\mathrm{d}w}
\newcommand{\du}{\,\mathrm{d}u}

\newcommand{\dphi}{\,\mathrm{d}\varphi}
\newcommand{\dmu}{\,\mathrm{d}\mu}
\newcommand{\LBdata}{\psi}
\newcommand{\LBfunc}{\vartheta}

\newcommand{\tfreq}{\mathcal{s}}

\newcommand{\spatial}{\mathbf{x}}
\newcommand{\statObj}{\xi}
\newcommand{\dynR}{\mathcal{R}}

\DeclareMathOperator{\diver}{div}

\DeclareMathOperator{\sign}{\mathrm{sign}}
\newcommand{\R}{\mathbb{R}} % set of real numbers
\newcommand{\N}{\mathbb{N}} % set of positive integers
\newcommand{\X}{\mathbb{X}}
\newcommand{\Y}{\mathbb{Y}}

\SetAlgoNoLine

\providecommand{\keywords}[1]
{
  {\small	
  \textbf{{Keywords---}} #1}
}

\author[1]{Gesa Sarnighausen}
\author[2]{Thorsten Hohage}
\author[3]{Martin Burger}
\author[4,5]{Andreas Hauptmann}
\author[6]{Anne Wald}

\affil[1]{Institute for Numerical and Applied Mathematics, University of G\"ottingen, Germany (\texttt{g.sarnighausen@math.uni-goettingen.de})}
\affil[2]{Institute for Numerical and Applied Mathematics, University of G\"ottingen, Germany (\texttt{hohage@math.uni-goettingen.de})}
\affil[3]{Helmholtz Imaging, Deutsches Elektronen-Synchrotron DESY, Notkestr. 85, Hamburg, 22607,
Germany (\texttt{martin.burger@desy.de})}
\affil[4]{Research Unit of Mathematical Sciences, University of Oulu, Finland (\texttt{andreas.hauptmann@oulu.fi})}
\affil[5]{Department of Computer Science, University College London, UK}
\affil[6]{Institute for Numerical and Applied Mathematics, University of G\"ottingen, Germany (\texttt{a.wald@math.uni-goettingen.de})}

\title{Regularization for time-dependent inverse problems: Geometry of Lebesgue-Bochner spaces and algorithms}
\date{{\em In memory of Alfred K.~Louis.}}        

\begin{document}

\maketitle

\begin{abstract}
\noindent
We consider time-dependent inverse problems in a mathematical setting using
Lebesgue-Bochner spaces. Such problems arise when one aims to recover a function from given observations where the function or the data depend on
time. Lebesgue-Bochner spaces allow to easily incorporate the different nature of time and space.

In this manuscript, we present two different regularization methods in Lebesgue-Bochner spaces:
\begin{enumerate}
    \item classical Tikhonov regularization in Banach spaces,
    \item temporal variational regularization by penalizing the time-derivative.
\end{enumerate} 
In the first case, we additionally investigate geometrical properties of Lebesgue-Bochner spaces. This particularly includes the calculation of the duality mapping and it is shown that these spaces are smooth of power type. The resulting Tikhononv regularization in Lebesgue-Bochner spaces is implemented using different regularities for time and space. 
Both methods are tested and evaluated for dynamic computerized tomography.

\end{abstract}

%\tableofcontents
\section{Introduction}

%\todo[inline]{write introduction}
Time-dependent or dynamic inverse problems (DIPs) have become an important research interest in recent years. The applications
range from dynamic imaging to the identification of time-dependent parameters in partial differential equations for material characterization. 

% Such problems arise when one aims to recover a function from given observations where the function or the data depend on time.

 A linear dynamic inverse problem can be expressed by the operator equation
\begin{align*}
    A(\vartheta) = \psi, \quad A: \X \to \Y,
\end{align*}
where the linear forward operator $A$, the target $\vartheta$ or the data $\psi$ can depend on time. This time-dependence can, e.g., be expressed by an appropriate choice of the function spaces $\X$ and $\Y$. Lebesgue-Bochner spaces (see Definition~\ref{def:Bochner}) allow to incorporate the different nature of time and space in a natural way, as discussed in, e.g., \cite[Chapter 23]{zeidler2013linear}.

There are two different approaches in DIPs: either reconstructing the object of interest for only one time-point, or reconstructing a time series of the moving object. In the first case, the data is inconsistent due to the time-dependence of the object. In the latter case, the problem is heavily under-sampled since in general there are only a few measurements per time step available.
In this article, we examine the second case by considering Lebesgue-Bochner spaces $L^p(0,T;L^r(\Omega))$ for a bounded spatial domain $\Omega \subset \R^n, n\in \N$ and $1 < p,r < \infty$ as function spaces $\X$ and $\Y$. By consequence, we have time-dependent data of dimension $n + 1$, %where $n$ is the spatial dimension. 
Similarly, we have a time-dependent target and forward operator.

Research on dynamic inverse problems is often linked to a concrete application, see~\cite{book_TDPIIP,A1-KLEIN;ET;AL:21}, such as dynamic computerized tomography~\cite{Burger_2017, bh2017, bh2014}, magnetic resonance
imaging~\cite{Hammernik2018}, emission tomography~\cite{hashimoto2019dynamic}, magnetic particle imaging~\cite{albers2023timedependent, bathke2017improved, A1-KALTENBACHER;ET;AL:21, knopp2012magnetic}, or structural health monitoring~\cite{A1-KLEIN;ET;AL:21, C2-LAMBWAVES-BUCH:18}.

Alfred K.~Louis, together with some of his students and collaborators, has played a pioneering role in the development of variational regularization techniques for dynamic inverse problems, including their efficient solution and applications to real-world problems. 
In particular, \cite{schmitt2002efficient1,schmitt2002efficient2} present a systematic approach in a time-discrete setting, which can be interpreted as a time-discretized version of the methods we aim to explore in this paper. Naturally, extending such variational regularization techniques to time-continuous settings involving time derivatives leads to the consideration of regularization in Lebesgue–Bochner spaces.

%Pioneering work on the variational regularization of dynamic inverse problems, its efficient solution, as well as applications, has been pioneered indeed by Alfred Louis with some his students and coworkers. In particular \cite{schmitt2002efficient1,schmitt2002efficient2} discuss a systematic approach in a time-discrete setting, which can indeed be understood as a time-discretized version of approaches we want to discuss in this paper. Naturally, proceeding to time-continuous versions in variational regularizations involving time derivatives leads us to the question of regularization in Lebesgue-Bochner spaces. 

Many methods for solving general dynamic inverse problems include additional information on the motion into the regularization scheme, e.g., some kind of motion model or a fixed state of the object serving as a template~\cite{bh2014,Burger_2017, Gris_2020, Hahn_2014}. 
It is also possible to consider the dynamic model an inexact version of the static model~\cite{bhw20, Nitzsche22, Goedeke_2023, masterarbeit}.
An overview of variational approaches to solve DIPs in a semi-discretized setting for image reconstruction with a main focus on methods including parametrized motion models is given in~\cite{Hauptmann2021}.   Since the motion is often not known, we aim to include the time-dependence only implicitly without the explicit use of a motion model.
\cite{Arridge22DIP} considers reconstruction techniques for fully discretized inverse problems that take the different nature of time and space into account without assuming an explicit motion model. 

In contrast, in~\cite{Burger_2024} the authors develop a continuous theory for DIPs in Lebesgue-Bochner spaces in accordance with the classical inverse problems theory. In particular, they extend the classical definition of ill-posedness and regularization for inverse problems in Banach spaces to DIPs in Lebesgue-Bochner spaces. To underline this, time-dependent regularization methods are analyzed theoretically but a practical application is still missing. In this article, we support the research of~\cite{Burger_2024} by presenting concrete regularization methods with practical applications. 

In particular, we aim to include the nature of time into regularization methods for general linear dynamic inverse problems mainly by using Lebesgue-Bochner spaces.
We propose two different approaches:
\begin{enumerate}
    \item Tikhonov regularization in Lebesgue-Bochner spaces where the difference of the temporal and spatial variable can be expressed by using different exponents for time and space in the Lebesgue-Bochner space,
    \item adding the time-derivative as a penalty term to the Tikhonov functional leading to a regularized solution that favors small changes in time.
\end{enumerate}

The aim of the first approach is to adapt a classical Tikhonov regularization algorithm in general Banach spaces to Lebesgue-Bochner spaces. This includes the nature of time only in the function space setting and does not require any conditions or previous knowledge on the movement. Thus, we do not expect great improvement of this regularization method compared to a static setting. However, it is interesting to examine how the different choices of exponents for time and space in the Lebesgue-Bochner setting influence the reconstructions.

To implement this approach, we make use of the dual method~\cite[Section 5.3.2]{RegularizationBanachSpaces} that minimizes the Tikhonov functional by gradient descent in the dual space for Banach spaces that are smooth of power type, i.e., they fulfill certain geometric smoothness conditions. To adapt this method for Lebesgue-Bochner spaces, we first investigate the geometrical properties of these spaces. In particular, we show that Lebesgue-Bochner spaces are smooth of power type and derive their duality mappings.

The second approach aims to include the nature of time not only by the choice of function spaces but additionally by penalizing the norm of the time derivative. For this, we assume that the object behaves 'smoothly' in time and postulate a certain regularity of the time derivative. In contrast to the first approach, we do not consider a Banach space setting, but a Hilbert space setting. This allows the use of more efficient algorithms compared to the Banach space setting.

For this approach, we derive a scheme that minimizes a variational problem penalizing the time derivative. We solve this problem by gradient descent using a two-step method where the first step is a Landweber step in Lebesgue-Bochner spaces minimizing the data discrepancy. The second step is responsible for the penalty terms and is only given implicitly by an $(n+1)-$dimensional PDE which we solve in Fourier space using the eigenfunctions of the underlying operator.

To evaluate these approaches in practice, we apply them both to dynamic computerized tomography. % as an example problem.

\emph{Outline.} 
The article is structured as follows. The first part (Section~\ref{sec:geom}) covers the geometry of Lebesgue and Lebesgue-Bochner spaces. After the basics of Lebesgue-Bochner spaces are introduced in Section~\ref{sec:lebesgue-bochner}, we derive the duality mappings in Lebesgue-Bochner spaces (Section~\ref{sec:dualityMapping}). Then we consider convexity and smoothness of Banach spaces in Section~\ref{sec:convexitySmoothness} where our main result (Lemma~\ref{lem:XuRoachCorollaryPsmooth}) is the proof of a consequence of the Xu-Roach inequalities. Finally, we apply these results to show that Lebesgue and Lebesgue-Bochner spaces are smooth of power type (Section~\ref{sec:smoothnessLebesgue(Bochner)}).
With the theoretical findings of Section~\ref{sec:geom}, we derive two regularization algorithms in the context of time-dependent inverse problems: classical Tikhonov regularization in Lebesgue-Bochner spaces (Section~\ref{sec:TikhonovDual}) and a regularization method penalizing the time derivative (Section~\ref{sec:TimeDerivative}). The last part, Section~\ref{sec:NumApplicationDynCT}, contains various numerical experiments with an application in dynamic computerized tomography.

\emph{Notation.}
In the following, let $X$ be a Banach space, $X^*$ its dual space, $\Omega \subset \R^n$ a bounded domain and $(S,\mathcal{A},\mu)$ a finite measure space.
In this article, we will omit the index indicating the space of norms if it is clear from the context which space is meant, i.e., we will write $\| x \|$ instead of $\| x \|_X$ for $x \in X$. The unit sphere in $X$ will be denoted by $S_X$, i.e., $S_X = \{ x \in X : \| x \| = 1 \}$. For an exponent $p$ with $1 < p < \infty$, its conjugate exponent $p^*$ is defined by $\frac{1}{p} + \frac{1}{p^*} = 1$.
Functions from Lebesgue-Bochner spaces will be denoted by $\LBfunc, \LBdata$, other functions or functionals by $f,g$ and elements from a general Banach space $X$ by $x$ and $y$.

\section{Geometry of Lebesgue and Lebesgue-Bochner spaces}
\label{sec:geom}

We are interested in minimizing a linear Tikhonov functional in Banach spaces.
 One algorithm to do this is the dual method presented in~\cite[Section 5.3.2]{RegularizationBanachSpaces}. 
In the next sections, we are interested in adapting this method for Lebesgue and Lebesgue-Bochner spaces.

To implement the dual method, see Section~\ref{sec:TikhonovDual}, we need to determine the duality mapping and a constant $G^X_{s}$ from the Definition~\ref{def:smoothness} of smoothness of power-type
for Lebesgue-Bochner spaces. We start with some definitions.

%In the following, we compute the duality mapping for Lebesgue-Bochner spaces. Further, we show that these spaces are smooth of power-type and determine a possible constant $G^X_s$. We start with some definitions. 

\subsection{Bochner spaces}
\label{sec:lebesgue-bochner}

\begin{definition}
\label{def:Bochner}
    The \textbf{Lebesgue-Bochner space} $L^p(S;X)$ for $1 \leq p < \infty$ is the linear space of all strongly $\mu$-measurable functions $\LBfunc :S\to X$ with $\int_S \lVert \LBfunc \rVert_X^p \dmu < \infty$. See~\cite{AnalysisBanachSpaces} for a definition and discussion of strong measurability.
%\begin{alignat*}{2}
%        &\int_S \lVert f \rVert_X^p \dmu < \infty &\qquad &\text{if }1 \leq p < \infty, \\
%        &\exists r \geq 0: \mu (  \{ \lVert f \rVert_X > r \} ) = 0 & &\text{if }p = \infty. 
%    \end{alignat*}
    The space $L^p(S;X)$ is a Banach space equipped with the norm%s
    \begin{align*}
        \lVert \LBfunc \rVert_{L^p(S;X)} &= \left( \int_S \lVert \LBfunc \rVert_X^p \dmu \right)^{\frac{1}{p}} \quad 1 \leq p < \infty %\\
       % \lVert f \rVert_{L^\infty(S;X)} &= \inf \Bigl\{ r \geq 0: \mu \{ \lVert f \rVert_X > r \} = 0 \Bigr\}
    \end{align*}
\end{definition}

A function $f : S \to X$ is an element of $L^p(S;X)$ if and only if $\lVert f \rVert_X : S \to \R$ is an element of $L^p(S)$.
For the convergence of regularization algorithms in Banach spaces, it is helpful if the Banach space be reflexive. Thus we are interested in what way the duality results for Lebesgue spaces transfer to Bochner spaces. 

\begin{theorem}
\label{th:dualityBochner}
    If $X$ is reflexive or $X^*$ is separable, it holds $  \left( L^p(S;X) \right)^* = L^{p*}(S; X^*)$
  %  \begin{align*}
  %      \left( L^p(S;X) \right)^* = L^{p*}(S; X^*)
  %  \end{align*}
    for adjoint exponents, i.e., $\frac{1}{p} + \frac{1}{p^*} = 1$ and $1 < p < \infty$.
\end{theorem}

%If $X$ is reflexive, it follows from Theorem~\ref{th:dualityBochner} that $L^p(S;X)$ is also reflexive analogous to the scalar Lebesgue space case. 
%In dynamic inverse problems the spaces $X$ and $Y$ in $L^p(S;X)$ and $L^q(S;Y)$ are likely to be reflexive because they are the function spaces of the static inverse problem and reflexivity of the function spaces is a condition to get convergence for regularization methods. Thus, if the static problem is posed in a reasonable setting, this transfers to the dynamic problem.
 A proof can be found in~\cite[Corollary 1.3.13]{AnalysisBanachSpaces}. In the following, we only consider reflexive spaces $X$ and can therefore use the duality result from Theorem~\ref{th:dualityBochner}.
We finish this section by stating an embedding result for Lebesgue-Bochner spaces.
\begin{lemma}
\label{lem:embeddingBochner}
    For $1 < p \leq q \leq \infty$ and $1 < r \leq s \leq \infty$ we have the continuous embedding $L^q(S;X) \subseteq  L^p(S;X)$ and $L^q(S;L^s(\Omega)) \subseteq L^p(S;L^r(\Omega))$ where $\Omega \subset \mathbb{R}^n$, $n \in \mathbb{N}$ is finite. 
\end{lemma}

\begin{proof}
    We use $\lesssim$ for relations up to multiplicative constants independent of $\LBfunc$.\\
    Let $\LBfunc \in L^q(S;L^s(\Omega))$. Then we have
    \begin{align*}
        \lVert \LBfunc \Vert_{L^p(S;L^r(\Omega))} 
        &= \left( \int_S \lVert \LBfunc(t) \rVert_{L^r(\Omega)}^p \dt \right)^{\frac{1}{p}} 
        \lesssim \left( \int_S \lVert \LBfunc(t) \rVert_{L^s(\Omega)}^p \dt \right)^{\frac{1}{p}}
        \lesssim \left( \int_S \lVert \LBfunc(t) \rVert_{L^s(\Omega)}^q \dt \right)^{\frac{1}{q}} 
        = \lVert \LBfunc \Vert_{L^q(S;L^s(\Omega))} 
    \end{align*}
    where the first inequality follows by the Rellich-Kondrachov embedding theorem for Lebesgue spaces, see, e.g.~\cite[Chapter 6]{adams2003sobolev}, and the second inequality follows by this embedding theorem applied to the norm $\lVert \LBfunc(t) \rVert_{L^s(\Omega)}$ which is a function in $L^q(S)$.

    The first embedding follows analogously by only using the embedding theorem for Lebesgue spaces once.
\end{proof}

\subsection{Duality mapping}
\label{sec:dualityMapping}

In this section, we derive the duality mapping in Lebesgue-Bochner spaces, which is an important tool in regularization in Banach spaces, see also \cite{RegularizationBanachSpaces}. To this end, we shortly recap a few necessary definitions and statements for general Banach spaces. 

%\begin{definition}{(gauge function)}
%    A continuous strictly increasing function $\phi$: $\R^+ \to \R^+$ fulfilling 
%    \begin{align*}
%        \phi(0) = 0, \quad \lim_{x \to \infty} \phi(x) = \infty
%    \end{align*}
%    is called \textbf{gauge function}.
%\end{definition}

\begin{definition}{(duality mapping)}
\label{def:dualityMapping}
    Let $X$ be a normed space. % and $\phi$ a gauge function. Then the set-valued \textbf{duality map} $ j_\phi : X \rightrightarrows X^*$ of the space $X$ with respect to the gauge function $\phi$ is defined as
    %\begin{align}
    %\label{eq:def_dualityMap}
    %   J^X_\phi x \coloneqq \{ x^* \in X^* : \langle x , x^* \rangle = \lVert x \rVert \lVert x^* \rVert, \lVert x^* \rVert = \phi(\lVert x \rVert) \}.
    %\end{align}
    %If $\phi_q(t) = t^{q-1}$ for $q \geq 1$, the duality map is denoted by $J^X_q \coloneqq J^X_{\phi_q}$. 
        Then the set-valued \textbf{duality map} $ J^X_q : X \rightrightarrows X^*$ of the space $X$ with respect to the function $h \mapsto h^{q-1}$ for $q \geq 1$ is defined as
    \begin{align}
    \label{eq:def_dualityMap}
       J^X_q (x) \coloneqq \{ x^* \in X^* : \langle x , x^* \rangle = \lVert x \rVert \lVert x^* \rVert, \lVert x^* \rVert = \lVert x \rVert^{q-1} \}.
    \end{align}
    For $q = 2$ the duality map $J^X_2$ is called the normalized duality map.
    A single-valued selection of the duality mapping is denoted by $j^X_q$.
\end{definition}

To compute the duality map $J^X_q$, we can use the subdifferential which is defined for convex functionals. For more details on geometric properties of Banach spaces we refer to~\cite[Chapter II]{AnalysisBanachSpaces} or \cite{chidume2008geometric}.

%\begin{definition}{(convex functional)}
%\label{def:convexFunctional}
%    Let $X$ be a linear space. A real-valued functional $f: X \to \R \cup \{ \infty \}$ is called \textbf{convex} if
%    \begin{align*}
%        f\big(\lambda x_1 + (1-\lambda) x_2\big) \leq \lambda f(x_1) + (1- \lambda) f(x_2)
%    \end{align*}
%    for all $x_1, x_2 \in X$ and all $0 \leq \lambda \leq 1$, i.e., that the line segment connecting two points of the graph lies always above the graph.
%\end{definition}

\begin{definition}{(subgradient, subdifferential)}
\label{def:subgradient}
    Let $f: X \to \R \cup \{\infty\}$ be a convex functional. Then a functional from the dual space $x^* \in X^*$ is called a \textbf{subgradient} of $f$ in $x$ if %it holds
    \begin{align*}
        f(y) \geq f(x) + \langle x^*, y-x \rangle
    \end{align*}
    for all $y \in X$. The set of all subgradients is denoted by $\partial f (x)$ and is called the \textbf{subdifferential}. %However, as it is common in literature the subdifferential will often be called subgradient if it is clear from the context which one is meant.
\end{definition}

Asplund's theorem combines the subgradient with the duality map:

\begin{theorem}{(Asplund's theorem)}
\label{thm:asplund}
    For a normed space $X$ and $q \geq 1$, the duality map $J^X_q$ can be computed by $ J^X_q = \partial \left( \frac{1}{q} \lVert \cdot \rVert^q \right).$
  %  \begin{align*}
  %      J^X_q = \partial \left( \frac{1}{q} \lVert \cdot \rVert^q \right).
  %  \end{align*}
\end{theorem}
A proof can be found, e.g., in~\cite{subdifferentialConvexFunctionals, Asplund}.

If the functional $f$ is Gateaux-differentiable, the subdifferential is single-valued and coincides with the gradient, i.e., $\partial f = \{ \nabla f \}$, see~\cite[Chapter 1, §2]{cioranescu1990geometry}. In the following if not specified otherwise we will only consider smooth norms, such that $j_q^X$ is unique. 
%For simplicity, we will refer to the unique single-valued selection $j_q^X$ as $J_q^X$ from now on.

It is easy to determine $j_q^X$ if $j_p^X$ is already known:
\begin{lemma}
\label{lem:dualityPtoQ}

\begin{enumerate}
    \item     For $p,q > 1$ and $x \in X$ it holds $j^X_q(x) = \|x \|^{q-p} j_p^X(x)$.
   % \begin{align*}
   %     j^X_q(x) = \|x \|^{q-p} j_p^X(x)
   % \end{align*}

    \item For a scalar $a$, $x \in X$ and $s \geq 1$ we have $ j_s^X(ax) = a^{s-1} j_s^X(x)$.
        %\begin{align*}
        %    j_s^X(ax) = a^{s-1} j_s^X(x)
        %\end{align*}
\end{enumerate}
\end{lemma}

\begin{proof}
The proof of 1. can be found in~\cite[Proposition 4.7.(f), Chapter 1]{cioranescu1990geometry}. From the definition of the duality map~\eqref{eq:def_dualityMap} we get the equation
\begin{align*}
     \langle ax, j_s^X(ax) \rangle = \| a x \|_X  \| j_s^X(ax) \|_{X^*} = a \| x \| \| a x \|^{s-1} = a^s \| x \|^s.
\end{align*}
%for arbitrary $a,x \in X$ and $s \geq 1$.
Since we also have
\begin{align*}
    \langle ax, a^{s-1} j_s^X(x) \rangle = a \| x \| a^{s-1} \| x \|^{s-1} = a^s \| x \|^s,
\end{align*}
by the definition of the duality map we get the equality $a^{s-1} j_s^X(x) = j_s^X(ax)$ since $x$ and $a$ were arbitrary.
\end{proof}

The duality mapping for $L^r(\Omega)$ spaces is single-valued and given by
\begin{align}
\label{eq:dualityLp}
    \langle  j_p(f),g \rangle = \int_\Omega \left(\| f \|_{L^r(\Omega)}^{p-r} |f(x)|^{r-1} \sign(f(x))\right) g(x) \dx, 
\end{align}
 for $f,g \in L^r{(\Omega)}$. A proof for the normalized duality mapping can be found in~\cite[Section 3.2]{chidume2008geometric}. The general duality mapping then follows with Lemma~\ref{lem:dualityPtoQ}.

Now, we can compute the duality mapping of $L^p(0,T;X)$ with respect to the duality mapping of the space $X$.

\begin{theorem}{(Duality mapping of $L^p(0,T;X)$)}
\label{thm:dualityLp(0,T;X)}
    Let $X$ be a normed space and $q \geq 1$. Let further either $X$ be reflexive or $X^*$ be separable. Then the duality mapping $j_q^{L^p(0,T;X)} \in L^{p*}(0,T;X^*)$ is given by
    \begin{align*}
        \left(j_q^{L^p(0,T;X)}\LBfunc \right)(t) = \lVert  \LBfunc \rVert_{L^p(0,T;X)}^{q-p} j^X_p \LBfunc (t)
    \end{align*}
    for $\LBfunc  \in L^p(0,T;X)$, $t \in [0,T]$.
\end{theorem}
\begin{proof}
    The duality map $j_q^{L^p(0,T;X)}$ is an element of the dual space $[ L^p(0,T;X) ]^*$, i.e., it is a functional $L^p(0,T;X) \to \R$. Let now $\psi \in L^p(0,T;X)$. With Asplund's theorem~\ref{thm:asplund} and the chain rule we get
    \begin{align*}
        \left( j_q^{L^p(0,T;X)}\LBfunc  \right) (\psi) 
        &= \partial \left( \frac{1}{q} \lVert \LBfunc  \rVert_{L^p(0,T;X)}^q \right)(\psi)
        =  \partial \left(  \frac{1}{q} \left( \int_0^T \lVert \LBfunc (t) \rVert^p_X \dt\right)^{\frac{q}{p}} \right) (\psi) \\
        &= \lVert \LBfunc  \rVert_{L^p(0,T;X)}^{q-p}   \int_0^T \langle \partial  \frac{1}{p} \lVert \LBfunc (t) \rVert^p_X, \psi(t) \rangle_{X^*,X} \dt  
        = \lVert \LBfunc  \rVert_{L^p(0,T;X)}^{q-p} \int_0^T \langle j_p^X \LBfunc (t), \psi(t) \rangle_{X^*,X} \dt.
    \end{align*}
    Since $L^{p^*}(0,T;X^*)$ can be identified with the dual space $[L^p(0,T;X) ]^*$, see Theorem~\ref{th:dualityBochner}, we can also express the duality mapping by $ j_q^{L^p(0,T;X)}\LBfunc  = \lVert \LBfunc  \rVert_{L^p(0,T;X)}^{q-p} j^X_p \LBfunc (\cdot).$
 %   \begin{align*}
 %       j_q^{L^p(0,T;X)}\LBfunc  = \lVert \LBfunc  \rVert_{L^p(0,T;X)}^{q-p} j^X_p \LBfunc (\cdot).
 %   \end{align*}
\end{proof}

\subsection{Convexity and smoothness of Banach spaces}
\label{sec:convexitySmoothness}

We want to show that Lebesgue-Bochner spaces are smooth of power-type since this is a prerequisite to use Tikhonov regularization in these spaces. First, we introduce the concepts of convexity and smoothness in Banach spaces and derive a consequence of the Xu-Roach inequalities for $s$-smooth spaces.

\begin{definition}{(Convexity)}
\label{def:convexity}
    For a Banach space $X$ we define the \textbf{modulus of convexity} $\delta_X : [0,2] \to [0,1]$ as $ \delta_X(\varepsilon) \coloneqq \inf \left\{ 1 - \| \frac{1}{2} (x + y) \| : \| x \| = \| y \| = 1, \| x - y \| \geq \varepsilon \right\}.$
   % \begin{align*}
   %     \delta_X(\varepsilon) \coloneqq \inf \left\{ 1 - \| \frac{1}{2} (x + y) \| : \| x \| = \| y \| = 1, \| x - y \| \geq \varepsilon \right\}.
   %\end{align*}

    The space $X$ is called
    \begin{itemize}
      %  \item \textbf{strictly convex} if $\| \frac{1}{2} (x  + y ) \| < 1$ for
%all $x,y$ from the unit sphere of $X$ satisfying the condition $x \neq y$.
        \item \textbf{uniformly convex} if $\delta_X(\varepsilon) > 0$ for all $\varepsilon$ with $0 < \varepsilon \leq 2$.
        \item \textbf{convex of power-type \textit{s}} or \textbf{\textit{s}-convex} if there exists a constant $c^X_s > 0$ such that for all $x,y \in X$ %and all $j_s^X \in J_s^X$, 
        \begin{align*}
            \| x - y \|^s \geq \| x \|^s - s\langle j_s^X(x), y \rangle + c_s \| y \|^s.
        \end{align*}
    \end{itemize}
\end{definition}

\begin{definition}{(Smoothness)}
\label{def:smoothness}
    For a Banach space $X$ we define the \textbf{modulus of smoothness} $\rho_X(\varepsilon) : [0, \infty) \to [0, \infty)$ as $\rho_X(\varepsilon) \coloneqq \frac{1}{2} \sup \bigl\{ 
        \lVert x + y \rVert + \lVert x - y \rVert - 2 : \lVert x \rVert = 1, \lVert y \rVert \leq \varepsilon
        \bigr\}.$
    %\begin{align*}
    %    \rho_X(\varepsilon) \coloneqq \frac{1}{2} \sup \bigl\{ 
    %    \lVert x + y \rVert + \lVert x - y \rVert - 2 : \lVert x \rVert = 1, \lVert y \rVert \leq \varepsilon
     %   \bigr\}.
    %\end{align*}
    The space $X$ is called
    \begin{itemize}
        \item \textbf{smooth} if for every $x \in X \setminus \{0\}$ there exists a unique $x^* \in X^*$ with $\lVert x^* \rVert = 1$ and $\langle x^*, x \rangle = \lVert x \rVert$.
        \item \textbf{uniformly smooth} if $\lim_{\varepsilon \to 0} \frac{\rho_X(\varepsilon)}{\varepsilon} = 0$.
        \item \textbf{smooth of power-type \textit{s}} or \textbf{\textit{s}-smooth} if there exists a constant $G^X_s > 0$ such that for all $x,y \in X$ %and all $j_s^X \in J_s^X$ 
        %we have
        \begin{align}
        \label{eq:defiSmoothnessPowerTypeInequality}
            \lVert x - y \rVert^s \leq \lVert x \rVert^s - s\langle j_s^X(x), y \rangle + G^X_s \lVert y \rVert^s.
        \end{align}
    \end{itemize}
\end{definition}

\begin{remark}
    The constants $c^X_s$ and $G^X_s$ from Definition~\ref{def:convexity} and~\ref{def:smoothness} cannot depend on $x,y \in X$. In the following, this assumption will hold for all appearing constants.
\end{remark}

 Due to statement 4 in the following Lemma~\ref{lem:smoothness}, smoothness and convexity can be seen as dual concepts. Therefore it is easy to see that each property of a smooth Banach space has an equivalent in a convex Banach space. Here, we focus only on smooth Banach spaces since we are interested in computing the constant $G^X_s$ from the definition of smoothness of power-type in Section~\ref{sec:smoothnessLebesgue(Bochner)}.
 Next, we summarize a few properties of smooth Banach spaces.
\begin{lemma} 
 \label{lem:smoothness}
 In a Banach space $X$ the following statements hold:
\begin{enumerate}
    \item If $X$ is $q$-smooth, it is also uniformly smooth and $q \leq 2$.
    \item If $X$ is uniformly smooth, $X$ is reflexive and smooth.
    \item $X$ is smooth if and only if every duality mapping $J_p^X$ is single-valued.
    \item $X$ is $s$-smooth/ uniformly smooth if and only if $X^*$ is $s^*$-convex/ uniformly convex.
    \item If and only if $X$ is $s$-smooth, it holds $\rho_X(\varepsilon) \leq C \varepsilon^s$ for a constant $C$.
\end{enumerate}
    
\end{lemma}
The statements and proofs can be found in~\cite[Chapter 2.3.1]{RegularizationBanachSpaces} and the therein cited references.

\begin{remark}
\label{rem:modulusSmoothnessLimitBehaviour}
    In 5. of Lemma~\ref{lem:smoothness} it suffices to consider the behavior close to $0$ since we can bound the modulus of smoothness by
    \begin{align*}
        \rho_X(\varepsilon) &= \frac{1}{2} \sup \bigl\{ 
        \lVert x + y \rVert + \lVert x - y \rVert - 2 : \lVert x \rVert = 1, \lVert y \rVert \leq \varepsilon
        \bigr\} 
        \leq \frac{1}{2} \sup \bigl\{ 
        2\lVert x  \| + 2\| y \rVert - 2 : \lVert x \rVert = 1, \lVert y \rVert \leq \varepsilon
        \bigr\} 
        \leq \varepsilon \leq \varepsilon^q
    \end{align*}
    for $\varepsilon \geq 1$ and $q > 1$. If we have $\rho_X(\varepsilon) \leq C \varepsilon^q$ for $\varepsilon < d < 1$, we can also find a constant $C$ for $d \leq \varepsilon < 1$ since $\rho$ is continuous and non-decreasing, see e.g.~\cite[1.e]{lindenstraussClassicalBanachspaces}.
\end{remark}

\subsubsection{Xu-Roach inequalities}
\label{sec:Xu-Roach}

Xu and Roach proved the following inequalities for uniformly smooth Banach spaces in 1991~\cite{XuRoach}.

\begin{lemma}{(Xu-Roach inequality)}
\label{lem:Xu-RoachSmooth}
    For any $1 < p < \infty$ and a Banach space $X$, the following statements are equivalent:
    \begin{enumerate}
        \item $X$ is uniformly smooth with modulus of smoothness $\rho_X$.
        \item $J_p^X$ is single-valued and it holds
        \begin{align}
            \label{eq:Xu-RoachUniformSmoothDualityMapping}
            \lVert j_p^X(x) - j_p^X(y) \rVert_{X^*} \leq
            G'_p\max(\lVert x \rVert_X , \lVert y \rVert_X )^{p} \lVert x - y \rVert^{-1}_X \rho_X\left( \frac{\lVert x - y \rVert_X}{\max(\lVert x \rVert_X, \lVert y \rVert_X )} \right) \quad \forall x,y \in X.
        \end{align}
        \item  \vspace{-3ex}
        \begin{align}
        \label{eq:Xu-RoachUniformSmoothNorm}
        \text{For all $x,y \in X$ it holds } \quad \lVert x - y \rVert^p \leq \lVert x \rVert^p - p\langle j_p^X(x),y \rangle + \tilde{\sigma}_p(x,y)
    \end{align}
    \begin{align}
    \label{eq:Xu-RoachSigma}
    \text{with }
\quad        \tilde{\sigma}_p(x,y) = p G'_p \int_0^1 \frac{\max( \lVert x - hy \rVert, \lVert x \rVert )^p}{h} \rho_X\left( \frac{h \lVert y \rVert}{\max( \lVert x - hy \rVert, \lVert x \rVert)} \right) \dH.
    \end{align}
         %   \item Inequality~\eqref{eq:Xu-RoachUniformSmoothNorm} holds for one $j_p^X \in J_p^X$.
    \end{enumerate}
\end{lemma}
    The proof and a possible choice for $G'_p$ is given in~\cite{XuRoach}.

%Similarly to Lemma~\ref{lem:ThorstenKazimierski} we can also show the following lemma using the results of Sprung~\cite{2019Sprung} which allows another characterization of $s$-smoothness of power-type, see Lemma~\ref{lem:SprungPowerType}.

To extend the Xu-Roach inequalities to $s$-smooth spaces, we need the following results of Sprung~\cite{2019Sprung} which allow another characterization of $s$-smoothness of power-type:

\begin{lemma}
\label{lem:SprungPowerType}
    A Banach space $X$ is $s$-smooth iff
    \begin{align}
    \label{eq:SprungCharacterization}
        \frac{1}{p} \| x -y \|^p - \frac{1}{p} \| x \|^p + \langle j_p^X(x),y \rangle \leq C \| y \|^s
    \end{align}
    for $1 < p < \infty$, a constant $C > 0$, $x,y \in X$ with $\|x \|=1$ and $\| y \| \leq d$ for $d > 0$. 
\end{lemma}

    This characterization extends Definition~\ref{def:smoothness} where $s$-smoothness was defined by $\lVert x - y \rVert^s \leq \lVert x \rVert^s - s\langle j_s^X(x), y \rangle + G^X_s \lVert y \rVert^s.$
  %  \begin{align*}
  %      \lVert x - y \rVert^s \leq \lVert x \rVert^s - s\langle j_s^X(x), y \rangle + G^X_s \lVert y \rVert^s.
  %  \end{align*}
    It is a direct consequence of %Proposition 4.3\todo{brauche ich die überhaupt?} and 
    Theorem 4.1(2) and Corollary 4.5 from~\cite{2019Sprung}:

\begin{comment}
\begin{lemma}[Sprung, 2019]
    Let $X$ be a Banach space, $1 < p < \infty$, $C > 0$, $\phi : \R^+ \to \R^+$ be nondecreasing, $0 \neq x \in X$ and $y \in X$. If we have
    \begin{align}
    \label{eq:SprungVoraussetzungI}
         \langle j_p^X(x) - j_p^x(y) , x - y \rangle \leq C \max \{\|x\| , \| y \| \}^p  \phi \left( \frac{ \| x - y \|}{\max \{ \| x \|, \| y \| \} } \right),
    \end{align}
    then there exists a constant $d > 0$, such that
    \begin{align*}
        \Delta^{j_p(x)}_{\frac{1}{p} \| \cdot \|}(x -y , x) 
        \coloneqq \frac{1}{p} \| x -y \|^p - \frac{1}{p} \| x \|^p + \langle j_p^X(x),y \rangle 
        \leq C \| x \|^p \phi \left( \frac{ \| y \|}{ \| x \| } \right),
    \end{align*}
    holds for all $x,y$ with $\| y \| \leq d \| x \| $ where $\Delta^{j_p(x)}_{\frac{1}{p}  \| \cdot \|}$ is the Bregman divergence.
\end{lemma}
\end{comment}

\begin{lemma}[Sprung, 2019]
\label{lem:Sprung2}
    Let $X$ be a Banach space and $1 < p < \infty$. If, for $d > 0$, $\varepsilon \leq d$ and all $x \in S_X$, we have
    \begin{align}
        \label{eq:SprungVoraussetzungII}
        \sup_{\| y \| = \varepsilon } \left | \frac{1}{p} \| x -y \|^p - \frac{1}{p} \| x \|^p + \langle j_p^X(x),y \rangle  \right | \leq \varepsilon^s,
    \end{align}
    we obtain the estimate $\rho_X(\varepsilon) \leq p^{1/p - 1} \varepsilon^s + C \varepsilon^2$
    %\begin{align*}
     %   \rho_X(\varepsilon) \leq p^{1/p - 1} \varepsilon^s + C \varepsilon^2
    %\end{align*}
    for the modulus of smoothness $\rho_X(\varepsilon)$ for $\varepsilon \leq d$, i.e., $X$ is $s$-smooth according to Remark~\ref{rem:modulusSmoothnessLimitBehaviour}.
\end{lemma}

\begin{lemma}[Sprung, 2019]
    If the space $X$ is $s$-smooth, then there exist constants $C_d > 0$ for all $d > 0$, such that $ \frac{1}{p} \| x -y \|^p - \frac{1}{p} \| x \|^p + \langle j_p^X(x),y \rangle  
       \leq C_d \| y \|^s$
    %\begin{align}
    %    \frac{1}{p} \| x -y \|^p - \frac{1}{p} \| x \|^p + \langle j_p^X(x),y \rangle  
    %    \leq C_d \| y \|^s
    %\end{align}
    holds for all $x,y \in X$ with $\| x \| = 1$ and $\| y \| \leq d$.
\end{lemma}

\begin{comment}
\begin{proof}[Proof of Lemma~\eqref{lem:SprungPowerType}]
    We start by assuming that~\eqref{eq:SprungCharacterization} holds. Let $d,C > 0$ be constants, $\| y \| \leq d \| x \| $ and $x \in S_X$
    Then we get
    \begin{align}
    \label{eq:bregDivIneq}
        \Delta^{j_p(x)}_{\frac{1}{p} \| \cdot \|}(x -y , x) 
        = \frac{1}{p} \| x -y \|^p - \frac{1}{p} \| x \|^p + \langle j_p^X(x),y \rangle 
        \leq C \| y \|^s
    \end{align}
    Thus we can estimate with~\eqref{eq:bregDivIneq} for $\varepsilon \leq d$
    \begin{align*}
        \sup_{\| y \| = \varepsilon } \left | \Delta^{j_p(x)}_{\frac{1}{p} \| \cdot \|}(x -y , x) \right |
        \leq \sup_{\| y \| = \varepsilon } \left | C \| y \|^s \right |
        = C \varepsilon^s.
    \end{align*}
    With $\phi(\varepsilon) = \varepsilon^s$ and Lemma~\ref{lem:Sprung2} this gives us the estimate 
    \begin{align*}
        \rho_X(\varepsilon) \leq p^{1/p - 1} \varepsilon^s + C \varepsilon^2 \in \mathcal{O}(\varepsilon^s)
    \end{align*}
    for $\varepsilon \leq \min(d,1)$ which shows that $X$ is $s$-smooth according to Remark~\ref{rem:modulusSmoothnessLimitBehaviour}.

    Let now $X$ be $s$-smooth. \todo[inline]{Beweis vervollständigen}
\end{proof}
\end{comment}

Lemma~\ref{lem:SprungPowerType} allows us to prove the equivalence of several statements for $s$-smooth Banach spaces as a consequence of the Xu-Roach inequalities as formulated in Lemma \ref{lem:XuRoachCorollaryPsmooth}. This statement %Lemma~\ref{lem:XuRoachCorollaryPsmooth} 
has already been published in the literature without proof, see, e.g.,~\cite{RegularizationBanachSpaces, KazimierskiSmoothnessBesov}. It was stated that the respective proof %of Lemma~\ref{lem:XuRoachCorollaryPsmooth} 
was already given by Xu and Roach in their original paper~\cite{XuRoach}. However, \cite{XuRoach} only contains a part of the proof of this consequence. In particular, the direction ($4. \Rightarrow 1. $) for $p \neq s$ is missing, which we prove using the characterization of power-type from Lemma~\ref{lem:SprungPowerType}. %in two different ways with lemmata~\ref{lem:ThorstenKazimierski} and~\ref{lem:SprungKazimierski}.
\begin{lemma}
\label{lem:XuRoachCorollaryPsmooth}
    Let $X$ be uniformly smooth. Then the following are equivalent for generic constants $C > 0$:
    \begin{enumerate}
        \item $X$ is $s$-smooth.
        \item The duality mapping $J_p^X$ is single-valued for some $p$ with $1 < p < \infty$ and for all $x,y \in X$ it holds
        \begin{align}
            \label{eq:XuRoachq-smoothDualityMap}
        \lVert j_p^X(x) - j_p^X(y) \rVert_{X^*}
        \leq C \max( \lVert x \rVert_X, \lVert y \rVert_X)^{p-s} \lVert x -y \rVert^{s-1}_X.
        \end{align}
        \item Statement (2.) holds for all $p$ with $1 < p < \infty$.
        \item For some $p$ with $1 < p < \infty$ %and some $j_p^X \in J_p^X$ 
        we have
        \begin{align*}
            \lVert x - y \rVert^p \leq \lVert x \rVert^p - p\langle j_p^X(x),y \rangle + \tilde{\sigma}_p(x,y)
        \end{align*}
        for all $x,y \in X$. Furthermore, we can bound $\tilde{\sigma}_p$ from~\eqref{eq:Xu-RoachSigma} by $ \tilde{\sigma}_p(x,y) \leq C p \int_0^1 h^{s-1} \max( \lVert x - hy \rVert , \lVert x \rVert)^{p-s} \lVert y \rVert^s \dH.$
       % \begin{align*}
       %     \tilde{\sigma}_p(x,y) \leq C p \int_0^1 h^{s-1} \max( \lVert x - hy \rVert , \lVert x \rVert)^{p-s} \lVert y \rVert^s \dH.
       % \end{align*}
        \item Statement (4.) holds for all $p$ with $1 < p < \infty$. % and all $j_p^X \in J_p^X$.
    \end{enumerate}
\end{lemma}

\begin{remark}
\label{rem:ConstantForP=S}
        If we know the constant $C$ in inequality~\eqref{eq:XuRoachq-smoothDualityMap} for $p=s$, we can choose $G^X_s = C$ in the definition of power-type~\eqref{eq:defiSmoothnessPowerTypeInequality}.
\end{remark}

\begin{proof}[Proof of Lemma~\ref{lem:XuRoachCorollaryPsmooth}]
\textbf{Step 1:} We start by showing the equivalence $1. \Leftrightarrow 5.$

        If $X$ is $s$-smooth, it is also uniformly smooth and each duality mapping is single-valued. By combining the Xu-Roach inequality~\eqref{eq:Xu-RoachUniformSmoothNorm} with the property 5. from Lemma~\ref{lem:smoothness}, we obtain
    \begin{align*}
        \lVert x - y \rVert^{p} &\leq 
        \lVert x \rVert^{p} - p \langle j_{p}^{X}(x) , y \rangle +
        p C \int_0^1 \frac{\max(\lVert x - hy \rVert, \lVert x \rVert)^{p}}{h} \frac{h^{s} \lVert y \rVert^{s}}{\max(\lVert x - hy \rVert, \lVert x \rVert)^{s}} \dH \\
        &= \lVert x \rVert^{p} - p \langle j_{p}^{X}(x) , y \rangle +
        p C \lVert y \rVert^{s} \int_0^1 h^{s - 1} \max(\lVert x - hy \rVert, \lVert x \rVert)^{p-s}\dt 
    \end{align*}
    for any $p$ with $1 < p < \infty$ which shows that inequality~\eqref{eq:defiSmoothnessPowerTypeInequality} holds.

    If we assume $5.$, we get $1.$ by choosing $p = s$ which yields
    \begin{align*}
        \lVert x - y \rVert^{s} &\leq 
        \lVert x \rVert^{s} - s \langle j_{s}^{X}(x) , y \rangle +
        s C \int_0^1 \frac{\max(\lVert x - hy \rVert, \lVert x \rVert)^{s}}{h} \frac{h^{s} \lVert y \rVert^{s}}{\max(\lVert x - hy \rVert, \lVert x \rVert)^{s}} \dH \\
        &= \lVert x \rVert^{s} - s \langle j_{s}^{X}(x) , y \rangle +
        s C \lVert y \rVert^{s} \int_0^1 h^{s - 1} \dH 
        = \lVert x \rVert^{s} - s \langle j_{s}^{X}(x) , y \rangle +
        C \lVert y \rVert^{s},
    \end{align*}
    and $X$ is $s$-smooth by definition.

 \textbf{Step 2:}   Next, we show the equivalence of $1,3,2$ and $4$.

        The implication $1. \Rightarrow 3.$ follows analogously to the implication $1. \Rightarrow 5.$ by applying inequality~\eqref{eq:Xu-RoachUniformSmoothDualityMapping} and 5. from Lemma~\ref{lem:smoothness}.
        The implication $3. \Rightarrow 2.$ is obvious.

    The implication $2. \Rightarrow 4.$ can be shown analogously as in the proof of Lemma~\ref{lem:Xu-RoachSmooth}~\cite{XuRoach}:
%Since $X$ is uniformly smooth, inequality~\eqref{eq:Xu-RoachUniformSmoothNorm} holds. We want to show the bound for $\tilde{\sigma_p}$. 
Since $X$ is uniformly smooth, $j_p^X$ is continuous from the norm topology of $X$ into the weak-$*$ topology of $X^*$~\cite{diestel2006geometryBanach}. Then the function $\Phi(h) = \| x  - hy \|^p$ with $1 < p < \infty$ is continuously differentiable and has the derivative $\Phi'(h) = -p \langle j^X_p(x - hy), y \rangle$. Then we get
\vspace{-2ex}
\begin{align*}
    \| x - y \|^p - \| x \|^p + p \langle j_p^X(x), y \rangle 
    &= \Phi(1) - \Phi(0) - \Phi'(0) 
    = \int_0^1 \Phi'(h) \dH - \Phi'(0) 
    = p \int_0^1 \langle -j_p^X(x - hy) + j_p^X(x), y \rangle \dH \\
    &\leq p    \int_0^1 \| j_p^X(x - hy) - j_p^X(x) \| \| y\| \dH 
    \leq p   C  \int_0^1 \max(\| x - hy \|, \|x \|)^{p-s} \| hy \|^{s-1} \| y\| \dH \\
    &= p   C  \int_0^1 h^{s-1} \max(\| x - hy \|, \|x \|)^{p-s} \|y \|^{s} \dH 
\end{align*}
    It remains to show the implication $4. \Rightarrow 1.$ This is trivial if $p = s$. Let us now assume $p \neq s$.
    We then have 
        \begin{align*}
            \frac{1}{p}\lVert x - y \rVert^p- \frac{1}{p}\lVert x \rVert^p + \langle j_p^X(x),y \rangle  
            &\leq   C \int_0^1 h^{s-1} \max( \lVert x - hy \rVert , 1)^{p-s} \lVert y \rVert^s \dH \\
            &\leq  C \int_0^1 h^{s-1}  (1 + hd)^{p-s} \| y \|^s \dH 
            \leq C \max((1 + d)^{p-s}, 1) \frac{1}{s} \| y \|^s ,
        \end{align*}
        for $\|x \| = 1$, and $\|y\| \leq d$ for $d > 0$, which shows the characterization of $s$-smoothness from Lemma~\ref{lem:SprungPowerType}.
\end{proof}

\subsubsection{A specification of smoothness of power-type for Banach spaces}
\label{sec:constantG_p}
%\subsubsection{Determining the constant $G^X_p$ for Bochner spaces}

%\todo[inline]{find better title for subsection?}

Lemma~\ref{lem:XuRoachCorollaryPsmooth} characterizes $s$-smoothness in different ways. However, we only obtain the constant $G^X_s$ from the definition of power-type~\ref{def:smoothness} from the equivalent inequality~\eqref{eq:XuRoachq-smoothDualityMap} for $p=s$, see Remark~\ref{rem:ConstantForP=S}.

In Section~\ref{sec:smoothnessLebesgue(Bochner)} we show that Lebesgue-Bochner spaces are smooth of power-type. To determine the corresponding constant $G^X_s$, we need to determine $G^X_s$ also for $p \neq s$ in~\eqref{eq:XuRoachq-smoothDualityMap} which we will derive in this Section.

\begin{theorem}
\label{lem:GqFromGp}
    If the Banach space $X$ is $s$-smooth with constant $G^X_s$ as in Definition \ref{def:smoothness}, it is also $q$-smooth for all $1 < q < s$. 
    For the $q$-smoothness we can choose the constant $G^X_q$ as $G^X_q = 2^{s-q} \max(2^s, G^X_s + K 2^{s-2})$ where $K$ is the Lipschitz constant of the function $f: [\frac{1}{2} , 1] \to \R$ with $h \mapsto h^{q-s}$.
\end{theorem}

To prove Theorem~\ref{lem:GqFromGp} we need an auxiliary Lemma, the proof of which can be found in the appendix.

\begin{lemma}
    \label{lem:ThorstenKazimierski}
    Let $s,q \in (1,\infty)$ and $X$ be a smooth Banach space. Then the following statements are equivalent:
    \begin{enumerate}[label=(\Alph*)]
            \item The duality mapping is Hölder continuous, i.e., $\exists L > 0 \quad\forall x,y \in X: \quad \| j_s(x) - j_s(y) \|_{X^*} \leq L \| x - y \|_X^{s-1}.$
               % \begin{align*}
               %     \exists L > 0 \quad\forall x,y \in X: \quad \| j_s(x) - j_s(y) \|_{X^*} \leq L \| x - y \|_X^{s-1}.
                %\end{align*}  
        \item For all $x,y \in X$ exists $D>0$ such that $\| j_q(x) - j_q(y) \|_{X^*} \leq D \max(\|x\|_X,\|y\|_X)^{q-s} \| x -y \|_X^{s-1}.$
        %\begin{align*} 
        %\exists D > 0 \quad \forall x,y \in X: \quad 
       % \| j_q(x) - j_q(y) \|_{X^*} \leq D \max(\|x\|_X,\|y\|_X)^{q-s} \| x -y \|_X^{s-1}.
        %\end{align*}
    \end{enumerate}
    If $L$ is given, $D$ can be chosen as $D = \max(2^s, L+ K2^{s-2})$ where $K$ is the Lipschitz constant of the function $f: [\frac{1}{2} , 1] \to \R$ with $h \mapsto h^{q-s}$. On the other hand, if $D$ is given, $L$ can be chosen as $L = \max(2^s, D + K' 2^{s-2})$ where $K'$ is the Lipschitz constant of the function $f: [\frac{1}{2} , 1] \to \R$ with $h \mapsto h^{s-q}$.
\end{lemma}

%The proof can be found in the appendix.

\begin{remark}
\label{rem:lipschitzConstantsK}
    We can easily compute the Lipschitz constants $K$ and $K'$ since the function $f:[\frac{1}{2},1] \to \R$ with $h \mapsto h^a$ for $a \in \R$ is continuous and differentiable. Thus, since the derivative is bounded on $[\frac{1}{2},1]$, the Lipschitz constant $K_a$ is given by $\sup_{h \in [\frac{1}{2},1]} f'(h)$. This means
    \begin{align*}
        K_a = |a 2^{1-a}| \text{   for }  a < 1, \qquad
        K_a = a \text{   for }  a \geq 1.
    \end{align*}
\end{remark}
\begin{proof}[Proof of Theorem~\ref{lem:GqFromGp}]
    The proof without the constant can be found in~\cite[Corollary 2.43]{RegularizationBanachSpaces}.
    Since $X$ is $s$-smooth and $q > 1$ we get with 3. from Lemma~\ref{lem:XuRoachCorollaryPsmooth} the inequality
    \begin{align}
    \label{eq:ineq_proofOfGqFromGp}
                \lVert j_p^X(x) - j_p^X(y) \rVert_{X^*}
        \leq C \max( \lVert x \rVert_X, \lVert y \rVert_X)^{p-s} \lVert x -y \rVert^{s-1}_X
    \end{align}
    for arbitrary $p > 1$.
    Since we can choose $C = G^X_s = L$ for $p = s$ according to Remark~\ref{rem:ConstantForP=S}, we obtain 
    \begin{align*}
        C = \max(2^s, L + K 2^{s-2}) = \max(2^s, G^X_s + K 2^{s-2})
    \end{align*}
    for $p = q \neq s$ with Lemma~\ref{lem:ThorstenKazimierski}.
   Since $s-q > 0$, we can estimate the right-hand side of~\eqref{eq:ineq_proofOfGqFromGp} for $p = q$ by%get by estimating
    \begin{align*}
        \max( \lVert x \rVert, \lVert y \rVert)^{q-s} \lVert x -y \rVert^{s-1}
        &= \left(\frac{\| x -y\|}{\max( \lVert x \rVert, \lVert y \rVert)}\right)^{s-q} \| x - y\|^{q-1} \\
        &\leq \left(\frac{\| x \| + \|y\|}{\max( \lVert x \rVert, \lVert y \rVert)}\right)^{s-q}   \| x - y\|^{q-1}
        \leq 2^{s-q}  \| x - y\|^{q-1},
    \end{align*}
    Accordingly, we set %such that we can chose 
    $G^X_q := 2^{s-q} \max(2^s, G^X_s + K 2^{s-2})$.
\end{proof}

   \subsubsection{Smoothness results for Lebesgue spaces} 
    \label{sec:smoothnessLebesgue(Bochner)}

  % \paragraph{Lebesgue spaces}

   The Lebesgue spaces $L^r(\Omega)$ for a domain $\Omega \subset \R^n$ and $1 < r < \infty$
   are $\min(2,r)$-smooth and $\max(2,r)$-convex, see e.g.~\cite{chidume2008geometric, Hanner1956, lindenstraussClassicalBanachspaces, XuRoach}. Now, we consider $ L^{p}(\Omega)$ with a bounded domain $\Omega$ and $1 < p < \infty$. 

  % To find the constant $G^X_{p} $ from~\eqref{eq:XuRoachq-smoothDualityMap} for Lebesgue spaces we need the following lemma:

%\begin{lemma}
%\label{lem:inequality}
%    Let $a,b \geq 0$, $c,d \in \R$ and $0 < p \leq 1$. Then we have
%    \begin{alignat*}{2}
%        &\text{(i)} \quad  &\lvert a^p - b^p \rvert &\leq \lvert a - b \rvert^p \\
 %        &\text{(ii)} \quad & a^p + b^p &\leq 2^{1-p} (a + b)^p \\
 %        &\text{(iii)} \quad & \Bigl\lvert \sign(c) \lvert c \rvert^p - \sign(d) \lvert d \rvert^p \Bigr\rvert &\leq 2^{1-p} \lvert c - d \rvert^p
  %  \end{alignat*}
%\end{lemma}
%The proof can be found in Appendix~\ref{lem:inequalityAppendix}.

\begin{lemma}{(Constant $G^{L^p(\Omega)}_p$ for Lebesgue spaces)}
\label{lem:contantLebesgue}
    The constant $G^{L^p(\Omega)}_{p}$ from the definition of smoothness of power-type for the space $L^p(\Omega)$ can be chosen as
    \begin{align*}
        G^{L^p(\Omega)}_{p} = 2^{2-p}, \quad \text{ if } 1 < p \leq 2, \qquad
        G^{L^p(\Omega)}_2 = p-1 ,\quad \text{ if } p > 2
    \end{align*}
\end{lemma}
For $p > 2$ this is proven in~\cite[Corollary 2]{XU1991InequalitiesPsmoothness} and for $1 < p \leq 2$ it is %straight forward 
simple to show~\eqref{eq:XuRoachq-smoothDualityMap} for $p=s$ and $G^{L^p(\Omega)}_p = 2^{2-p}$. %\todo[inline]{es steht in kazimierski, und in schöpfer,louis,schuster ohne beweis. kazimierski zitiert zwei chinesische paper von xu für beweis. kurze erklärung zu beweis}

\subsubsection{Smoothness results for Lebesgue-Bochner spaces}
With the theoretical results of the previous Sections~\ref{sec:Xu-Roach} and~\ref{sec:constantG_p} we can now apply the theory to Lebesgue-Bochner spaces.
Uniform convexity and smoothness of Bochner spaces have already been considered in the literature:
    
    \begin{lemma}
    Let $X$ be reflexive. Then the Lebesgue-Bochner space $L^p(0,T;X)$ is uniformly convex/smooth if and only if $X$ is uniformly convex/smooth.
\end{lemma}
\begin{proof}
    The proof for uniform convexity can be found in Theorem 2 in~\cite{Day1941SomeMU}.
    Let now $X$ be uniformly smooth. Due to 5. in Lemma~\ref{lem:smoothness}, $X^*$ is uniformly convex and so is $L^{p^*}(0,T;X^*)$. Applying Lemma~\ref{lem:smoothness} again shows that $L^{p}(0,T;X)$ is uniformly smooth. Since all implications are equalities, the statement is true also for 'if and only if'.
\end{proof}

\begin{remark}
    %In the special case that the sequence space $\X$ is a direct sum of $p$-convex Banach spaces $(X_i)_{i=1}^\infty$ for $p \geq 2$ with norm $\| x \|_\X = \left(\sum_{i=1}^\infty \| x_i\|_{X_i}^q \right)^{\frac{1}{q}}$ where $x = (x_i)_{i=1}^n$ and $q < p$, it has been shown in~\cite[Proposition 1]{Figiel1972AnEO} that $\X$ is $p$-convex. By duality arguments we also get the according statement for smoothness of power-type. However the result does not hold for general Lebesgue-Bochner spaces.
    
    %Furthermore, 
    It is already known that there exists an equivalent norm in $L^q(0,T;X)$ where $X$ is $p$-smooth such that the space $L^q(0,T;X)$ is $\min(p,q)$-smooth~\cite[Section 3.7]{AnalysisBanachSpaces} (which is based on~\cite{Pisier1975MartingalesWV} and~\cite[Chapter 10]{Pisier2016MartingalesIB}). This approach involves martingales in Banach spaces. In the next Lemma we will prove that the space $L^q(0,T;X)$ is $\min(p,q)$-smooth already with the standard norm. 
\end{remark}

\begin{lemma}
\label{lem:psmoothnesLp(0,T;X)}
    Let $X$ be $p$-smooth with known constant $G_p^X$ and $1< q<\infty$. Then $L^q(0,T;X)$ is $\min(p,q)$-smooth with constant 
    \begin{alignat*}{2}
        G_{q}^{L^q(0,T;X)} &=  2^{p-q} \max(2^p, G^X_p + K 2^{p-2}), & \text{if } q = \min(p,q)\\
        G_p^{L^q(0,T;X)} &=  \max(2^q, 2^{q-p} G_p^X + K 2^{q-2}) + K' 2^{p-2}  & \text{if } p = \min(p,q),
    \end{alignat*}
    where $K$ and $K'$ are the Lipschitz constants of the functions $f: [\frac{1}{2} , 1] \to \R$ with $h \mapsto h^{q-p}$ and $g: [\frac{1}{2} , 1] \to \R$ with $h \mapsto h^{p-q}$ respectively, see Remark~\ref{rem:lipschitzConstantsK}.
\end{lemma}

\begin{proof}
    The case of Lebesgue-Bochner spaces is similar to the case of Besov spaces, which has been considered in~\cite{KazimierskiSmoothnessBesov}.

    \textbf{Step 1:} We show 2.~from Lemma~\ref{lem:XuRoachCorollaryPsmooth} and start with the case $q = \min(p,q)$. Due to Lemma~\ref{lem:GqFromGp}, the space $X$ is also $q$-smooth.
    Let $\LBfunc ,\LBdata  \in L^q(0,T;X)$. We apply 2.~from Lemma~\ref{lem:XuRoachCorollaryPsmooth} with Remark~\ref{rem:ConstantForP=S} to the space $X$ such that
    \begin{align*}
        \lVert j_q^X(\LBfunc (t)) - j_q^X(\LBdata (t)) \rVert_{X^*} \leq G^X_q \lVert \LBfunc (t) - \LBdata (t) \rVert_X^{q-1}
    \end{align*}
    for all $t \in [0,T]$. % with the constant $G^X_q$ from the Definition~\ref{def:smoothness} of smoothness of power-type for the space $X$.
    Now with the duality mapping for Lebesgue-Bochner spaces from Theorem~\ref{thm:dualityLp(0,T;X)} we have $\left( j_q^{L^q(0,T;X)}\LBfunc  \right)(t) = j_q^X(\LBfunc (t))$ and thus
    \begin{align*}
        \lVert j_q^{L^q(0,T;X)}(\LBfunc ) - j_q^{L^q(0,T;X)}(\LBdata ) \rVert^{q^*}_{[L^q(0,T;X)]^*} 
        &= \int_0^T \lVert j_q^X(\LBfunc (t)) - j_q^X(\LBdata (t)) \rVert^{q^*}_{X^*} \dt \
        \leq (G^X_q )^{q^*} \int_0^T \lVert \LBfunc (t) - \LBdata (t) \rVert^{q^*(q-1)}_X \dt \\
        &=  (G^X_q )^{q^*} \lVert \LBfunc  - \LBdata  \rVert^q_{L^q(0,T;X)}
    \end{align*}
    using $q^*(q-1) = q$. Now by taking the $q^*$-th root we arrive at
    \begin{align*}
        \lVert j_q^{L^q(0,T;x)}(\LBfunc ) - j_q^{L^q(0,T;X)}(\LBdata ) \rVert_{[L^q(0,T;X)]^*}
        \leq G^X_q  \lVert \LBfunc  - \LBdata  \rVert^{q-1}_{L^q(0,T;X)}
    \end{align*}
    which shows that $L^p(0,T;X)$ is $q$-smooth.
    %because $\frac{q}{q^*} = q-1$. 
    With Lemma~\ref{lem:XuRoachCorollaryPsmooth} we can choose $G_{q}^{L^p(0,T;X)} = G_q^X$. With Theorem~\ref{lem:GqFromGp} we have $G_q^X = 2^{p-q} \max( 2^p, G_p^X + K 2^{p-2}).$

    \textbf{Step 2:} Next, we consider the case $p = \min(p,q)$.
    We apply Lemma~\ref{lem:ThorstenKazimierski} and arrive at
    \begin{align*}
        \lVert j_q^X(x) - j_q^X(y) \lVert_{X^*} \leq \underbrace{\max(2^p, G_p^X + K 2^{p-2})}_{=: D} \max(\lVert x \rVert_X, \lVert y \rVert_X)^{q-p} \lVert x - y \rVert_X^{p-1}
    \end{align*} 
    \vspace{-2ex} for all $x,y \in X$.
    Then we get
    \begin{align*}
        \lVert j_q^{L^q(0,T;X)}(\LBfunc ) - j_q^{L^q(0,T;X)}(\LBdata ) \rVert_{L^{q^*}(0,T;X^*)}^{q^*} 
        & = \int_0^T \lVert j_q^X(\LBfunc (t)) - j_q^X(\LBdata (t)) \rVert^{q^*}_{X^*} \dt \\
        & \leq D^{q^*} \int_0^T \hspace{-2ex} \max\{\lVert \LBfunc (t) \rVert_X, \lVert \LBdata (t) \rVert_X \}^{(q-p)q^*} \lVert \LBfunc (t) - \LBdata (t) \rVert_X^{(p-1)q^*} \dt \\
        & \leq  D^{q^*} \hspace{-1ex} \left(\int_0^T \hspace{-2ex}  \max\{\lVert \LBfunc (t) \rVert_X, \lVert \LBdata (t) \rVert_X \}^q \dt \right)^{\frac{(q-p)q^*}{q}} \hspace{-1.5ex} \left( \int_0^T \hspace{-1ex}  \lVert \LBfunc (t) - \LBdata (t) \rVert_X^{q} \dt \right)^{\frac{(p-1)q^*}{q}} \hspace{-1ex} ,
    \end{align*}
    where we used H\"olders inequality with exponents $r = \frac{q}{q^*(q-p)}$ and $r^* = \frac{q}{q^*(p-1)} = \frac{q-1}{p-1} > 1$.
    Now we take the $q^*$-th root and arrive at
    \vspace{-2ex}
    \begin{align*}
        \lVert j_q^{L^q(0,T;X)}(\LBfunc ) - j_q^{L^q(0,T;X)}(\LBdata ) \rVert_{L^{q^*}(0,T;X^*)} 
        & \leq D  \left(\int_0^T \max\{\lVert \LBfunc (t) \rVert_X, \lVert \LBdata (t) \rVert_X \}^q \dt \right)^{\frac{q-p}{q}} \lVert \LBfunc  - \LBdata  \rVert^{p-1}_{L^q(0,T;X)} \\
        & \leq D  \left(\int_0^T \lVert \LBfunc (t) \rVert_X^q + \lVert \LBdata (t) \rVert_X^q \dt \right)^{\frac{q-p}{q}} \lVert \LBfunc  - \LBdata  \rVert^{p-1}_{L^q(0,T;X)} \\
        &= D  \left( \lVert \LBfunc  \rVert^q_{L^q(0,T;X)} + \lVert \LBdata  \rVert^q_{L^q(0,T;X)} \right)^{\frac{q-p}{q}} \lVert \LBfunc  - \LBdata  \rVert^{p-1}_{L^q(0,T;X)} \\
        & \leq D  \left( [\lVert \LBfunc  \rVert_{L^q(0,T;X)} + \lVert \LBdata  \rVert_{L^q(0,T;X)} ]^q \right)^{\frac{q-p}{q}} \lVert \LBfunc  - \LBdata  \rVert^{p-1}_{L^q(0,T;X)} \\
        %\leq &D  \left(  2^q \max\{\lVert \LBfunc  \rVert_{L^q(0,T;X)}, \lVert \LBdata  \rVert_{L^q(0,T;X)} \}^q \right)^{\frac{q-p}{q}} \lVert \LBfunc  - \LBdata  \rVert^{p-1}_{L^q(0,T;X)} \\
        & = D  2^{q-p} \max\{\lVert \LBfunc  \rVert_{L^q(0,T;X)}, \lVert \LBdata  \rVert_{L^q(0,T;X)} \}^{q-p} \lVert \LBfunc  - \LBdata  \rVert^{p-1}_{L^q(0,T;X)},
    \end{align*}
    where we used the inequality $a^{q} + b^{q} \leq
        (a + b)^{q}$ for $q > 1$ and $a,b \geq 0$.
    Thus $L^p(0,T;X)$ is $p$-smooth. By applying Lemma~\ref{lem:ThorstenKazimierski} again and with Lemma~\ref{lem:XuRoachCorollaryPsmooth}, we get
    \begin{align*}
        G_p^{L^q(0,T;X)} &= \max\left[ 2^p, D 2^{q-p} + K' 2^{p-2} \right] 
        = \max\left[ 2^p, 2^{q-p}\max(2^p, G_p^X + K 2^{p-2})  + K' 2^{p-2} \right] \\
        &= \max\left[ 2^p, \max(2^q, 2^{q-p} G_p^X + K 2^{q-2}) + K' 2^{p-2} \right] 
        = \max(2^q, 2^{q-p} G_p^X + K 2^{q-2}) + K' 2^{p-2}
    \end{align*}
    since $p \leq q$ and $K' 2^{p-2} > 0$.
\end{proof}

\begin{remark}
        The space $L^p(0,T;L^q(\Omega))$ is $\max(p,q,2)$-convex, since its dual space $L^{p^*}(0,T;L^{q^*}(\Omega))$ is $s^* =\min(p^*,q^*,2)$-smooth. Thus, with Lemma~\ref{lem:smoothness} the space $L^p(0,T;L^q(\Omega))$ is $s$-convex where $s = \max(p,q,2)$.     
        
        By combining Lemma~\ref{lem:contantLebesgue} and~\ref{lem:psmoothnesLp(0,T;X)} with Remark~\ref{rem:lipschitzConstantsK}, we can now compute the explicit constants for the Bochner space $L^p(0,T;L^q(\Omega))$.
\end{remark}

\section{Application to regularization algorithms}
\label{sec:RegAlgo}

After developing the geometry of Lebesgue-Bochner spaces, we now want to apply our findings to regularization algorithms. In particular, the theoretical results allow us to use Tikhonov regularization in Lebesgue-Bochner spaces with different exponents (Section~\ref{sec:TikhonovDual}). 

As another approach to include the time-dependency of the problem to the regularization scheme, we  develop a regularization algorithm in Section~\ref{sec:TimeDerivative} that penalizes the time derivative.

\subsection{Linear Tikhonov regularization in Banach spaces: A dual method}
\label{sec:TikhonovDual}

In this section we are interested in adapting the dual method~\cite[Section 5.3.2]{RegularizationBanachSpaces} for minimizing a linear Tikhonov functional in Banach spaces to Lebesgue and Lebesgue-Bochner spaces.
We minimize the Tikhonov functional
\begin{align}
\label{eq:tikhonov_func_banach}
    T_\alpha(\vartheta) \coloneqq \frac{1}{u} \lVert A \vartheta - \psi \rVert_\Y^u + \alpha \frac{1}{v} \lVert \vartheta \rVert_\X^v
\end{align}
for a linear operator $A : \X \to \Y$, $u,v > 1$, and a convex and smooth of power-type Banach space $\X$ and an arbitrary Banach space $\Y$. As explained in~\cite[Chapter 5.3]{RegularizationBanachSpaces} the minimizer $\vartheta_\alpha^{reg}$ of the Tikhonov functional is well-defined and provides a regularized solution %ation method.

With the duality mapping $J_v^\X$ (Definition~\ref{def:dualityMapping}), the subgradient $\partial T_\alpha$ (Definition~\ref{def:subgradient}) and the Bregman distance $D_{j_v^\X}$ defined by $D_{j_v^\X}(x,y) \coloneqq \frac{1}{v} \| x \|^v - \frac{1}{v} \| y \|^v - \langle j_v^\X(y), x-y \rangle$
%\begin{align*}
%    D_{j_v^\X}(x,y) \coloneqq \frac{1}{v} \| x \|^v - \frac{1}{v} \| y \|^v - \langle j_v^\X(y), x-y \rangle
%\end{align*}
for $x,y \in \X$ (see, e.g., \cite{RegularizationBanachSpaces} for details) and $j_v^\X$ a single-valued selection of the duality mapping $J_v^\X$, the dual method is given by the following algorithm:

\begin{algo}[Dual method]
\label{alg:dualMethod}

\begin{algorithm}
%\TitleOfAlgo{Dual method}

\KwIn{$\vartheta_0 \in \mathbb{X}$ initial point, $\vartheta_0^* = j_v^\mathbb{X}(\vartheta_0)$ dual initial point, \\
\hspace{1.8em} $G^{\X^*}_{v^*}$ constant from the Definition~\ref{def:smoothness} of smoothness of powertype, \\
\hspace{1.8em} $R_0$ with $D_{j_v^\mathbb{X}}(\vartheta_0, \vartheta_\alpha^{reg}) \leq R_0$}

$k \gets 0$\;

\While{$0 \notin \partial T_\alpha(\vartheta_k)$}{
  Choose $\theta_k \in \partial T_\alpha(\vartheta_k)$ \\
  $\mu_k \gets \min \left\{ \left( \frac{\alpha}{G^{\X^*}_{v^*}} \frac{R_k}{\| \theta_k \|^{v^*} }\right)^{\frac{1}{v^* - 1}}, \frac{1}{\alpha}  \right\}$ \\
  $R_{k+1} \gets (1 - \mu_k \alpha) R_k + \mu_k^{v^*} \frac{G^{\X^*}_{v^*}}{v^*} \| \theta_k \|^{v^*}$ \\
  $\vartheta_{k+1}^* \gets \vartheta_k^* - \mu_k \theta_k$ \\
  $\vartheta_{k+1} \gets j^{\mathbb{X}^*}_{v^*}(\vartheta_{k+1}^*)$
}

\KwOut{$\vartheta_{k+1}$}

\end{algorithm}
\end{algo}

\begin{comment}   
\begin{lstlisting}[caption={dual method}, label=alg1:dual, mathescape=true,lineskip=3.5pt,
  language=Pascal]
$\mathbf{Input}$: $ \text{initial point }  \vartheta_0 \in \X, \text{ dual initial point } \vartheta^*_0 = J_v^\X(\vartheta_0), \text{ constant } G_{v^*} \text{ from the definition}$
$\text{of smoothness of power-type for the space } \X^* \text{ and constant } R_0 \text{ fulfilling }  D_{j_v^\X}(\vartheta_0, x_\alpha^\delta) \leq R_0$.
n = $0$
while $0 \notin \partial T_\alpha(\vartheta_n)$
    $\text{choose } \theta_n \in \partial T_\alpha(\vartheta_n)$
    $\mu_n$ = $\min_{\mu > 0} \left\{ \max(0, 1-\mu \alpha) R_n + D_{j_{v^*}^{\X^*}} (\vartheta_n^* - \mu \theta_n, \vartheta_n^*) \right\} $
    $R_{n+1}$ = $(1-\mu_n \alpha) R_n + D_{j_{v^*}^{\X^*}} (\vartheta_n^* - \mu_n \theta_n, \vartheta_n^*)$
    $\vartheta_{n+1}^*$ = $\vartheta_n^* - \mu_n \theta_n$
    $\vartheta_{n+1}$ = $j^{\X^*}_{v^*}(\vartheta_{n+1}^*)$
$\mathbf{Output}$: $\vartheta_{n+1}$
\end{lstlisting}
\end{comment}

% Alternative mit Konstante
%$\mu_n$ = $\min \left\{ \left( \frac{\alpha R_n}{G_{v^*} \lVert \theta_n \rVert^{v^*} } \right)^{\frac{1}{v^* - 1}}, \frac{1}{\alpha}\right\} $
%    $R_{n+1}$ = $(1-\mu_n \alpha) R_n + \mu_n^{v^*} \frac{G_{v^*}}{v^*} \lVert \theta_n \rVert^{v^*}$

The convergence rates of the dual method depend on the additional geometric properties of the spaces $\X$ and $\Y$, see Table~\ref{tab:convergenceRatesDual}. This is interesting for our purposes since by varying the exponents of the Lebesgue or Lebesgue-Bochner spaces we get different convergence rates.

\begin{table}[h]
    \centering
    \begin{tabular}{|c|c|l|}
        \hline
        $\X$ & $\Y$  & convergence rate \\
        \hline
        $v$-convex & arbitrary  & $\| \vartheta_k - \vartheta_\alpha^{reg} \| \leq C k^{- \frac{1}{v(v-1)}}$ \\
        \hline
        $v$-convex &  $u$-convex & $\| \vartheta_k - \vartheta_\alpha^{reg} \| \leq C k^{- \frac{M-1}{[(M-1)(v-1)-1]v}}$ if $M \coloneqq \max(u,v) > 2$ \\
        \hline
        $2$-convex & $2$-convex & $\| \vartheta_k - \vartheta_\alpha^{reg} \| \leq C \exp^{-\frac{k}{C}}$\\
        \hline
    \end{tabular}
    \caption{Convergence rates for the dual method. The constant $C$ is generic and can be different for every case. The proof of the results can be found in~\cite[Section 5.3.2]{RegularizationBanachSpaces}.}
    \label{tab:convergenceRatesDual}
\end{table}

\begin{remark}
If we choose $\X$ and $\Y$ as Lebesgue or Lebesgue-Bochner spaces, the following statements hold:
\begin{itemize}
    \item     Since the duality mappings of $\X$ and $\Y$ are single-valued, c.f.~\eqref{eq:dualityLp} and Theorem~\ref{thm:dualityLp(0,T;X)}, the Tikhonov functional~\eqref{eq:tikhonov_func_banach} is G\^ateaux differentiable with derivative $  \partial T_\alpha(\vartheta) = \{\nabla T_\alpha(\vartheta)\} = \{A^* j_u^\Y ( A \vartheta - \psi) + \alpha j_v^\X(\vartheta)\}.$
   % \begin{align*}
   %     \partial T_\alpha(\vartheta) = \{\nabla T_\alpha(\vartheta)\} = \{A^* j_u^\Y ( A \vartheta - \psi) + \alpha j_v^\X(\vartheta)\}.
   % \end{align*}
    \item Since $\X$ and $\Y$ are convex of power-type or even $2$-convex, see Section~\ref{sec:smoothnessLebesgue(Bochner)}, we get the convergence rates of the second or third row of Table~\ref{tab:convergenceRatesDual} if we choose the spaces and the exponents of the Tikhonov functional $u$ and $v$ appropriately.
\end{itemize}
\end{remark}

\subsection{Temporal variational regularization}
\label{sec:TimeDerivative}

\subsubsection{Variational model for smooth motion}

%As we saw in the previous Section~\eqref{sec:NumApplicationTikhonovBochner} by using the Lebesgue Bochner setting we could improve the runtime and make the choice of parameters simpler but the reconstruction error was similar as in the static setting.
So far, we have purely exploited geometric properties of the function spaces. Now, we assume that the object moves continuously in time such that we can add penalization of the time derivative to the Tikhonov functional, obtaining a distinct temporal regularization. 

\paragraph{Function space setting} We stay in a Hilbert space setting and consider the spaces
\begin{align*}
\X =  L^2(0,T;L^2(\Omega)) \cap H^1(0,T;H^{-1}_0(\Omega)), \qquad
\Y = L^2(0,T;L^2(\Omega')), 
\end{align*}  
for rectangular bounded domains $\Omega, \Omega' \subset \R^n$, i.e., $\vartheta \in L^2(0,T;L^2(\Omega))$ and $\partial_t \vartheta \in L^2(0,T;H^{-1}_0(\Omega))$ if $ \vartheta \in \X$. Additionally, we assume homogeneous Neumann boundary conditions in time and space, i.e., $\partial_t \vartheta(0,\spatial) = \partial_t \vartheta(T,\spatial ) = 0$ for all $\spatial  \in \Omega$ and $\nabla \vartheta(t,\spatial) = 0$ for $\spatial  \in \partial \Omega$ and all $t \in [0,T]$.

To obtain some flexibility in the weighting of %the two components of 
the $H^1$-norm, we introduce the equivalent norm
\begin{align}
\label{eq:H-1GammaNorm}
    \| f \|^2_{H^1(\Omega)} = \gamma \| f \|^2_{L^2(\Omega)} + \| \nabla f \|^2_{L^2(\Omega)}
\end{align}
for $f \in H^1(\Omega)$, $\gamma > 0$ on $H^1(\Omega)$ and define $H^{-1}_0(\Omega)$ as the dual space of $H^1(\Omega)$ with respect to the norm~\eqref{eq:H-1GammaNorm}.

\paragraph{Motion models and time derivative}
\label{paragraph:timeDerivativeSpaceH-1}
The choice of the space $\X$ includes many time-dependent functions since the assumption $\partial_t \vartheta \in L^2(0,T;H^{-1}_0(\Omega))$ is fulfilled for a wide class of movements.
%\todo[inline]{Notation $T$ is overloaded: Changed to $\Gamma$}
Let $\Phi : [0,T] \times \R^n \to \R^n $ be a diffeomorphism that maps a point in $\R^n$ to a point in $ \R^n$ for each $t \in [0,T]$.
Then $\Gamma(t,\cdot) = \Phi^{-1}(t,\cdot)$ can be interpreted as a motion function.
The velocity $V$ of the motion $\Gamma$ is given by $V = \partial_t \Phi$.
    Let further $\statObj : \R^n \to \R$ be a function representing a static object. We consider two types of motion models:
    \begin{enumerate}
        \item[(a)] For a mass preserving motion, we can generate a time-dependent function $\vartheta$, representing the object $\statObj$ undergoing the motion $\Gamma$, by $\vartheta(t,\spatial ) = \statObj(\Gamma(t,\spatial )) | \det (\nabla \Gamma(t,\spatial )) |.$
        %\begin{align}
    %\label{eq:constructDynFunction_mass}
     %   \vartheta(t,\spatial ) = \statObj(\Gamma(t,\spatial )) | \det (\nabla \Gamma(t,\spatial )) |.
    %\end{align}
   % A function that preserves mass 
   which fulfills the continuity equation
        \begin{align}
        \label{eq:continuity}
        \partial_t \vartheta = -\diver (\vartheta V).
    \end{align}
        \item[(b)] For an intensity preserving motion, i.e., each particle preserves its intensity over time, the time-dependent function is given by $  \vartheta(t,\spatial ) = \statObj(\Gamma(t,\spatial )) $
       % \begin{align}
    %\label{eq:constructDynFunction_intensity}
    %    \vartheta(t,\spatial ) = \statObj(\Gamma(t,\spatial )) 
    %\end{align}
    which fulfills the optical flow constraint
        \begin{align}
        \label{eq:optical_flow}
        \partial_t \vartheta = - V \nabla \vartheta .
    \end{align}
    \end{enumerate}

Further details on motion models can, e.g., be found in~\cite[Section 2.1]{Hahn2021} or~\cite[Chapter 4]{dirks2015variational}.
Since we consider a bounded domain $\Omega \subset \R^n$, we chose $\Omega$ such that the investigated object does not leave $\Omega$ at any time $t \in [0,T]$.

\begin{lemma}
\label{lem:timeDerInH-1}
    Let $\vartheta$ be a time-dependent function that either conserves mass or intensity. If the velocity $V$ of the movement fulfills $V \in L^\infty(0,T;L^\infty(\Omega))$, it holds $\partial_t \vartheta \in L^2(0,T;H^{-1}_0(\Omega))$.
\end{lemma}

    If the mass is conserved, then the continuity equation $ \partial_t \vartheta = - \diver(\vartheta V)$
%\begin{align*}
%    \partial_t \vartheta = - \diver(\vartheta V)
%\end{align*}
is fulfilled and $\partial_t \vartheta \in L^2(0,T;H^{-1}_0(\Omega))$ holds if $\vartheta V \in L^2(0,T;L^2(\Omega))$. Since $\vartheta \in L^2(0,T;L^2(\Omega))$, it suffices that $V \in L^\infty(0,T;L^\infty(\Omega))$. This holds analogously for intensity preserving movements.

\begin{comment}
If the movement conserves intensity,
 the equation
\begin{align*}
   \partial_t \vartheta = - V \nabla \vartheta = - \diver(\vartheta V) + \vartheta \diver(V)
\end{align*}
is fulfilled. As in the first case it suffices that $V \in L^\infty(0,T;L^\infty(\Omega))$.
\end{comment}

\paragraph{Minimization problem of temporal variational regularization}
We aim at solving the minimization problem
\begin{align}
\label{eq:minProblemTimeDer}
\begin{split}
\min_{\vartheta \in \X} E( \vartheta) = &\min_{\vartheta \in \X} \left\{ \underbrace{\frac{1}{2} \| A \vartheta - \psi\|_{L^2(0,T;L^2(\Omega))}^2}_{=F(\vartheta)} + \underbrace{\frac{\alpha}{2} \| \vartheta \|^2_{L^2(0,T;L^2(\Omega))} + \frac{\beta}{2} \| \partial_t \vartheta \|^2_{L^2(0,T;H^{-1}_0(\Omega))}}_{=G(\vartheta)} \right\} \\
= &\min_{\vartheta \in \X} \left\{ \frac{1}{2} \int_0^T \left( \| (A \vartheta)(t) - \psi(t) \|_{L^2(\Omega)}^2 + \frac{\alpha}{2} \| \vartheta(t) \|^2_{L^2(\Omega)} + \frac{\beta}{2} \| \partial_t \vartheta(t) \|^2_{H^{-1}_0(\Omega)} \right) \dt \right\} ,
\end{split}
\end{align}
%\todo[inline]{Brauchen wir $\delta$? Zweite Zeile kein $\delta$., $\delta$ weg} 
for the data $\psi$ and regularization parameters $\alpha, \beta > 0$.
\begin{comment}
    
\end{comment}by defining $E(\vartheta) = F(\vartheta) + G(\vartheta)$
with
\begin{align*}
	F(\vartheta) &= \frac{1}{2}\| A \vartheta - \psi \|^2_{L^2(0,T;L^2(\Omega))}, \\
	G(\vartheta) &= \frac{\alpha}{2} \| \vartheta \|^2_{L^2(0,T;L^2(\Omega))} + \frac{\beta}{2} \| \partial_t \vartheta \|^2_{L^2(0,T;H^{-1}_0(\Omega))} . 
\end{align*}
\end{comment}
The following lemma yields a computationally convenient representation of the $H^{-1}_0$-norm:
\begin{lemma}[$H^{-1}_0$-norm]
    The $H^{-1}_0$-norm on the space $H^{-1}_0(\Omega)$ with homogeneous Neumann boundary conditions can be understood as
    \begin{align*}
        \lVert \phi \rVert^2_{H^{-1}_0(\Omega)} = \| \nabla (-\Delta_N + \gamma I)^{-1} \phi \|^2_{L^2(\Omega)} + \gamma \| (-\Delta_N + \gamma I)^{-1} \phi \|^2_{L^2(\Omega)},
    \end{align*}
    where $\Delta_N$ is the Neumann Laplace operator, $\phi \in H^{-1}_0(\Omega)$ and $\gamma > 0$.
\end{lemma}
\begin{proof}
Let $\phi \in H^{-1}_0(\Omega)$. Then by the Riesz representation theorem there is a unique function $f \in H^1(\Omega)$, s.t. 
\begin{align*}
    \phi(g)
    = \langle f, g \rangle_{H^1(\Omega)} 
    = \int_\Omega \nabla f \nabla g \dx + \gamma \int_\Omega f g \dx 
    = \int_\Omega g (- \Delta_N + \gamma I) f \dx 
    = \langle (-\Delta_N + \gamma I)f, g \rangle_{L^2(\Omega)}
    \quad \forall g \in H^1(\Omega),
\end{align*}
where %we have used the equivalent norm
%\begin{align}
%\label{eq:H-1GammaNorm}
%    \| g \|^2_{H^1(\Omega)} = \gamma \| g \|^2_{L^2(\Omega)} + \| \nabla g \|^2_{L^2(\Omega)}
%\end{align}
%for $\gamma \geq 0$ on $H^1(\Omega)$ and 
the second equality follows with Green's first identity and the Neumann boundary conditions. This shows that $(- \Delta_N + \gamma I) f \in H^{-1}_0(\Omega)$.
We set $ \phi =  (- \Delta_N + \gamma I) f$ and thus $f = (- \Delta_N + \gamma I)^{-1} \phi$. By consequence, we obtain the representation of the $H^{-1}_0$-norm as
\begin{align*}
	\| \phi \|^2_{H^{-1}_0(\Omega)} 
 &=  \| f \|^2_{H^1(\Omega)}
 = \| (-\Delta_N + \gamma I)^{-1} \phi \|^2_{H^1(\Omega)} \\
 &=  \| \nabla (-\Delta_N + \gamma I)^{-1} \phi \|^2_{L^2(\Omega)} + \gamma \| (-\Delta_N + \gamma I)^{-1} \phi \|^2_{L^2(\Omega)}.
\end{align*}
\end{proof}
\begin{lemma}[Existence and uniqueness of the minimizer]
    The minimizer $\vartheta$ of problem~\eqref{eq:minProblemTimeDer} exists, is unique and fulfills the first order optimality condition for the Fr\'echet derivative $E'$ of $E$:
\begin{align}
\label{eq:optimalityCond}
0 &= E'(\vartheta) = %F'(\vartheta)  +  G'(\vartheta) = 
A^*(A \vartheta - \psi) + \alpha \vartheta - \beta (-\Delta_N + \gamma I)^{-1} (\partial_{tt}  \vartheta).
\end{align}%\todo{$+\gamma I$ fehlt in (18)?} 
%where $E'$ is the Fr\'echet derivative of $E$.%, $F'$ and $G'$ are the Fr\'echet derivatives of $E$, $F$ and $G$ respectively.
\end{lemma}

\begin{proof}
It is clear that $E$, as a sum of convex and Fr\'echet differentiable norms, is again convex and Fr\'echet differentiable. Hence, a unique minimzer exists and fulfills the first order optimality condition~\eqref{eq:optimalityCond}.

   % Squared norms in Hilbert spaces are Fr\'echet differentiable, see e.g.~\cite[Example 2.3 in VII.2]{Ana2AmannEscher}, and norms are convex (homogeneity and triangle inequality). Since $E$ consists of a sum of norms, $E$ is Fr\'echet differentiable and convex. This means, see e.g.~\cite[Theorem 2.8.1, 2.8.4]{optimization}, that a unique solution exists which fulfills the first order optimality condition for the Fr\'echet derivative $E'$.

    %Since norms are convex (homogeneity and triangle inequality) and $E$ consists of a sum of norms, we have a convex minimization problem. This means, see e.g.~\cite[Theorem 2.8.1, 2.8.4]{optimization}, that a unique solution exists which fulfills the first order optimality condition for the gradient since $E$ is Fr\'echet-differentiable. \todo[inline]{Standardargument quadrierte Normen in Hilbert Räumen Frechetableitung}

    To compute $E'$, we compute the first variation for $\lambda \in \R$ and $\psi \in \X$ and evaluate it for $\lambda = 0$.
    We restrict ourselves to calculating the derivative of $G_2(\vartheta) \coloneqq \frac{\beta}{2} \| \partial_t \vartheta \|^2_{L^2(0,T;H^{-1}_0(\Omega))}$. %The derivation of the other terms is straightforward.
\begin{comment}
    We start with $ G = G_1 + G_2$ where 
    \begin{align*}
        G_1(\vartheta) &= \frac{\alpha}{2} \| \vartheta \|^2_{L^2(0,T;L^2(\Omega))},  \\
        G_2(\vartheta) &= \frac{\beta}{2} \| \partial_t \vartheta \|^2_{L^2(0,T;H^{-1}_0(\Omega))} 
    \end{align*}
and obtain
    \begin{align*}
    G_1'(\vartheta)h &= \frac{\partial}{\partial \lambda} \left[ \frac{\alpha}{2} \| \vartheta + \lambda h \|^2_{L^2(0,T;L^2(\Omega))}
         \right]_{\lambda = 0} \\
         &=  \frac{\partial}{\partial \lambda} \left[ \frac{\alpha}{2} ( \vartheta + \lambda h , \vartheta + \lambda h )_{L^2(0,T;L^2(\Omega))} 
         \right]_{\lambda = 0} \\
         &= \alpha (\vartheta, h)_{L^2(0,T;L^2(\Omega))}
    \end{align*}
\end{comment}
        \begin{align*}
    G_2'(\vartheta)\psi &= \frac{\partial}{\partial \lambda} \left[ \frac{\beta}{2} \| \partial_t (\vartheta + \lambda \psi) \|^2_{L^2(0,T;H^{-1}_0(\Omega))}
         \right]_{\lambda = 0} \\
         &=  \frac{\partial}{\partial \lambda} \Bigl[ \frac{\beta}{2} \left( \nabla (-\Delta_N + \gamma I)^{-1} (\partial_t \vartheta + \lambda \partial_t \psi), \nabla (-\Delta_N + \gamma I)^{-1} (\partial_t \vartheta + \lambda \partial_t \psi) \right)_{L^2(0,T;L^2(\Omega))} \\
         & \quad +   \frac{\beta \gamma}{2} \left(  (-\Delta_N + \gamma I)^{-1} (\partial_t \vartheta + \lambda \partial_t \psi),  (-\Delta_N + \gamma I)^{-1} (\partial_t \vartheta + \lambda \partial_t \psi) \right)_{L^2(0,T;L^2(\Omega))}
         \Bigr]_{\lambda = 0} \\
         &= \beta \left(\nabla (-\Delta_N + \gamma I)^{-1} \partial_t \vartheta, \nabla (-\Delta_N + \gamma I)^{-1} \partial_t \psi\right)_{L^2(0,T;L^2(\Omega))} \\
         & \quad + \beta \gamma \left( (-\Delta_N + \gamma I)^{-1} \partial_t \vartheta, (-\Delta_N + \gamma I)^{-1} \partial_t \psi\right)_{L^2(0,T;L^2(\Omega))}.
    \end{align*}
    By applying Green's first identity %\todo{prelim} 
    to the spatial domain and using the spatial Neumann boundary conditions we obtain
    \begin{align*}
        G_2'(\vartheta)\psi &=  \beta \left(\nabla (-\Delta_N + \gamma I)^{-1} \partial_t \vartheta, \nabla (-\Delta_N +\gamma I)^{-1} \partial_t \psi\right)_{L^2(0,T;L^2(\Omega))} \\
        & \quad + \beta  \left( (-\Delta_N + \gamma I)^{-1} \partial_t \vartheta, \gamma I (-\Delta_N + \gamma I)^{-1} \partial_t \psi\right)_{L^2(0,T;L^2(\Omega))} \\
        &= \beta \left( (-\Delta_N + \gamma I)^{-1} \partial_t \vartheta,  (-\Delta_N + \gamma I)(-\Delta_N + \gamma I)^{-1} \partial_t \psi\right)_{L^2(0,T;L^2(\Omega))} \\
        &= \beta \left( (-\Delta_N + \gamma I)^{-1} \partial_t \vartheta, \partial_t \psi\right)_{L^2(0,T;L^2(\Omega))}.
    \end{align*}
    Now, we apply partial integration to the time domain and get
    \begin{align*}
        G_2'(\vartheta)\psi &= \beta \left( (-\Delta_N + \gamma I)^{-1} \partial_t \vartheta, \partial_t \psi\right)_{L^2(0,T;L^2(\Omega))} \\
        &= \beta \left(-\partial_t (-\Delta_N + \gamma I)^{-1} \partial_t \vartheta, \psi\right)_{L^2(0,T;L^2(\Omega))} + \psi(T)(-\Delta_N + \gamma I)^{-1} \partial_t \vartheta(T) - \psi(0)(-\Delta_N + \gamma I)^{-1} \partial_t \vartheta(0).
    \end{align*}
    Assuming Neumann boundary conditions in time we obtain $G_2'(\vartheta) =  - \beta (-\Delta_N + \gamma I)^{-1} \partial_{tt} \vartheta.$
   % \begin{align*}
   %     G_2'(\vartheta) =  - \beta (-\Delta_N + \gamma I)^{-1} \partial_{tt} \vartheta.
   % \end{align*}
%    Similarly, we obtain
%    \begin{align*}
%        G_1'(\vartheta) &= \alpha \vartheta \\
%         F'(\vartheta) &= A^* (A \vartheta - \psi).
%    \end{align*}

%\begin{align*}
%       F'(\vartheta)h &=\frac{\partial}{\partial \lambda} \left[ \frac{1}{2} \| A (\vartheta + \lambda h) - \psi^\delta \|^2_{L^2(0,T;L^2(\Omega))} 
 %        \right]_{\lambda = 0} \\
 %      &=\frac{\partial}{\partial \lambda} \left[ 
 %      \frac{1}{2} ( A (\vartheta + \lambda h) - \psi^\delta, A (\vartheta + \lambda h) - \psi^\delta )_{L^2(0,T;L^2(\Omega))} \right]_{\lambda = 0}  \\
 %      &= \frac{1}{2} \left[ 2 \lambda \| A h \|^2_{L^2(0,T;L^2(\Omega))} 
 %      + 2(A \vartheta - \psi^\delta, A h)_{L^2(0,T;L^2(\Omega))} \right]_{\lambda = 0} \\
 %      &= (A \vartheta - \psi^\delta, A h)_{L^2(0,T;L^2(\Omega))} \\
 %      &= (A^* (A \vartheta - \psi^\delta), h)_{L^2(0,T;L^2(\Omega))} 
 %   \end{align*}

   % since $h \in L^2(0,T;L^2(\Omega)) \cap H^1(0,T;H^{-1}_0(\Omega))$ was chosen arbitrarily.
\end{proof}

\begin{comment}
            + \frac{\alpha}{2} \| \vartheta + \lambda h \|^2_{L^2(0,T;L^2(\Omega))} 
        + \frac{\beta}{2} \| \nabla (- \Delta_N)^{-1} \partial_t (\vartheta + \lambda h) \|    2_{L^2(0,T;L^2(\Omega))}
\end{comment}

\subsubsection{Minimization and algorithm}
%\todo[inline]{It seems this is for $d=2$ now mainly? @Gesa: n-dim cosine transform} 
The first order optimality condition~\eqref{eq:optimalityCond} is fulfilled if the minimizer is the solution of the PDE
\begin{align*}
\beta \partial_{tt} \vartheta = \alpha (- \Delta_N + \gamma I) \vartheta + (-\Delta_N + \gamma I) A^* (A \vartheta - \psi).
\end{align*}

This $(n+1)$-dimensional PDE is too large for solving the resulting linear system directly if the spatial dimension $n > 1$ or the resolution is high. Instead, we determine the minimizer of~\eqref{eq:optimalityCond} by solving the equivalent fixed point equation
\begin{align*}
 0 =  F'(\vartheta) +  G'(\vartheta) \quad
 \Leftrightarrow \quad
 \vartheta - \tau  F'(\vartheta) = \vartheta + \tau  G'(\vartheta)
\end{align*}
with a suitable step size $\tau > 0$.

To solve this fixed point equation, we use the implicit scheme
\begin{align}
\label{eq:algoScheme1}
\vartheta_{k + \frac{1}{2}} = \vartheta_k - \tau  F'(\vartheta_k) &= \vartheta_k - \tau A^* (A \vartheta_k - \psi) \\
\label{eq:algoScheme2}
\vartheta_{k+1} + \tau \ G'(\vartheta_{k+1}) &= \vartheta_{k+\frac{1}{2}}.
\end{align}

The first step is a simple gradient descent step equivalent to the standard Landweber iteration.
In the second step~\eqref{eq:algoScheme2} the next iterate $\vartheta_{k+1}$ is the solution of the PDE
\begin{align}
\label{eq:PDESecondStep}
 (1 + \tau \alpha) (-\Delta_N + \gamma I) \vartheta_{k+1} - \tau \beta \partial_{tt} \vartheta_{k+1} =  (-\Delta_N + \gamma I) \vartheta_{k+\frac{1}{2}}.
\end{align}

We solve~\eqref{eq:PDESecondStep} in Fourier space. Instead of using the full Fourier series, we only use cosine functions since they can incorporate Neumann boundary conditions and are the eigenfunctions of the operator given by the left side of~\eqref{eq:PDESecondStep}. 
In particular, we have the eigenfunctions $u_{w_i}$ and $u_\tfreq $ for $ \tfreq ,w_i \in \N$ for each dimension given by
\begin{align*}
   u_{\omega_i}(\spatial_i) = \cos\left(\omega_i \pi \frac{\spatial_i + b_i}{2b_i}\right), \quad
    u_\tfreq (t) = \cos\left(\tfreq  \pi \frac{t}{T}\right),
\end{align*}
for the spatial variables $\spatial_i \in [-b_i,b_i]$, $b_i > 0$ for $i =  1,\ldots,n$ and the temporal variable $t \in [0,T]$.

It is straight forward to see that
\begin{comment}
\begin{align*}
     (1 + \tau \alpha) (-\Delta_N + \gamma I) u - \tau \beta (\partial_{tt} u) = \lambda u
\end{align*}
for some $\lambda \in \R$ and any of the functions $u_{w_i}$ or $u_\tfreq $, which shows they are indeed eigenfunctions.

Also the product of these eigenfunctions $u_\tfreq  u_{\omega_1} \cdots u_{\omega_n}$ is an eigenfunction:
\end{comment}
\begin{align*}
          (1 + \tau \alpha) (-\Delta_N + \gamma I) u_\tfreq  u_{\omega_1} \cdots  u_{\omega_n}  - \tau \beta \partial_{tt}  (u_\tfreq  u_{\omega_1} \cdots u_{\omega_n}) 
         %&=          (1 + \tau \alpha) \left (u_\tfreq  u_{\omega_1} \cdots u_{\omega_n} 
         %- u_\tfreq  \sum_{i=1}^n \left(\partial{\spatial_i \spatial_i}\right) u_{\omega_1} \cdots u_{\omega_n}  \right)
         %- \tau \beta u_{\omega_1} \cdots u_{\omega_n} (\partial_{tt} u_\tfreq ) \\
         %&=          (1 + \tau \alpha) \left (u_\tfreq  u_{\omega_1} \cdots u_{\omega_n} 
         %- u_\tfreq  \sum_{i=1}^n u_{\omega_1} \cdots u_{\omega_n} (-1) \left( \frac{\omega_i \pi}{2 b_i}   \right)^2 \right)
         % - \tau \beta u_{\omega_1} \cdots u_{\omega_n} u_\tfreq  (-1) \left( \frac{\tfreq  \pi}{T} \right)^2 \\
          =         \left( (1 + \tau \alpha) + \sum_{i=1}^n \left( \frac{\omega_i \pi}{2 b_i}   \right)^2
          + \tau \beta \left( \frac{\tfreq  \pi}{T} \right)^2 \right) u_\tfreq  u_{\omega_1} \cdots u_{\omega_n},
\end{align*}
which shows that the product of these functions is indeed an eigenfunction.

Furthermore, they fulfill the Neumann boundary conditions, since we have
\begin{align*}
    \partial_{\spatial_i} u_{w_i}(b_i) = \partial_{\spatial_i} u_{w_i}(-b_i) = 0, \text{ and }  \partial_{t} u_{\tfreq }(0) = \partial_{t} u_{\tfreq }(T) = 0.
\end{align*}

Thus, we want to solve
\begin{align*}
  (1 + \tau \alpha) \mathcal{C}_{(\omega, \tfreq )} ((-\Delta_N + \gamma I) \vartheta_{k+1})
  - \tau \beta \mathcal{C}_{(\omega, \tfreq )} (\partial_{tt} \vartheta_{k+1}) 
  = \mathcal{C}_{(\omega, \tfreq )} \left((-\Delta_N + \gamma I) \vartheta_{k + \frac{1}{2}}\right),
\end{align*}
where $\mathcal{C}_{(\omega, \tfreq )}$ are the coefficients of the $(n+1)$-dimensional cosine series with time variable $\tfreq $ and space variable $\omega$, i.e.,
\begin{align}
\label{eq:cosineSeries}
    \vartheta(t,\spatial ) \approx \sum_{\tfreq  \in \N} \sum_{\omega_1 \in \N} \cdots \sum_{\omega_n \in \N} \mathcal{C}_{(\omega,\tfreq )} (\vartheta) u_\tfreq  u_{\omega_1} \cdots u_{\omega_n} %\cos\left(\tfreq  \frac{ \pi}{T}t\right)  \cos\left(\omega_1 \pi \frac{\spatial_1 + b_1}{2b_1}\right) \cdots  \cos\left(\omega_n \pi \frac{\spatial_n + b_n}{2b_n}\right),
\end{align}
with the spatial domain $\Omega = [-b_1,b_1] \times \ldots \times [-b_n,b_n]$, $b_i > 0$. %Since the time domain $[0,T]$ is not symmetric, we use the half frequency and get $\frac{\textbf{2} \pi}{T}t$ in~\eqref{eq:cosineSeries}.

The coefficients of the cosine series are then given by
\begin{align*}
    \mathcal{C}_{(\omega,\tfreq )} (\vartheta) = \frac{2}{T \prod_{i=1}^n b_i} \int_0^T \int_{-b_1}^{b_1} \cdots \int_{-b_n}^{b_n} \vartheta(t,\spatial ) u_\tfreq  u_{\omega_1} \cdots u_{\omega_n} \dx \dt.
\end{align*}

Using the property of the coefficients of the cosine series
\begin{align*}
\mathcal{C}_{(\omega, \tfreq )}(D^a \vartheta) =  (-1)^{|a|}  \left(\frac{\pi}{2b_1} \omega_1\right)^{a_1} \cdots \left(\frac{\pi}{2b_n} \omega_n \right)^{a_n} \left(\frac{ \pi \tfreq }{T} \right)^{a_{n+1}} \mathcal{C}_{(\omega, \tfreq )} (\vartheta)
\end{align*}
where $a_i$ is even for $i=1,\ldots,n+1$, we obtain
\begin{align}
\label{eq:FourierIterateEquation}
	(1+\tau \alpha) (|\omega'|^2 + \gamma) \mathcal{C}_{(\omega, \tfreq )}(\vartheta_{k+1}) +\tau \beta (\tfreq')^2 \mathcal{C}_{(\omega, \tfreq )}(\vartheta_{k+1}) &= (|\omega'|^2 + \gamma) \mathcal{C}_{(\omega, \tfreq )}(\vartheta_{k + \frac{1}{2}}),
\end{align}
where $\omega'$ is  given by $\omega'_i = \frac{\pi}{2b_i} \omega_i$ for $i = 1,\ldots,n$ and $\tfreq' = \frac{ \pi \tfreq }{T}$.

This means, we can compute the coefficients of the cosine series of $\vartheta_{k+1}$ by
\begin{align}
	 \mathcal{C}_{(\omega, \tfreq )} ( \vartheta_{k+1}) = \frac{1}{(1 + \tau\alpha)  + \tau \beta \frac{(\tfreq ')^2}{|\omega'|^2 + \gamma}} \mathcal{C}_{(\omega, \tfreq )}(\vartheta_{k + \frac{1}{2}}).
\end{align}

All in all, we arrive at the scheme:

\begin{algo}
\label{algo:timeDerivative}
For a given iterate $\vartheta_k$ the next iterate can be computed by
    \begin{align}
    \label{eq:algo_landweber}
        \vartheta_{k + \frac{1}{2}} &= \vartheta_k - \tau A^* (A \vartheta_k - \psi), \\
    \label{eq:algo_FourierCoef}
        \mathcal{C}_{(\omega, \tfreq )}(\vartheta_{k+1}) &=
            \frac{1}{(1 + \tau\alpha)  + \tau \beta \frac{(\tfreq ')^2}{|\omega'|^2 + \gamma}} \mathcal{C}_{(\omega, \tfreq )}(\vartheta_{k + \frac{1}{2}}) \quad \text{for all $\omega_i,\tfreq  \in \mathbb{N}$},\\
    \label{eq:algo_endStep}
        \vartheta_{k+1} &\text{ is given by~\eqref{eq:cosineSeries},}
    \end{align}
    where $\tau$ is chosen as $\tau = \frac{1}{\lVert A^* A \rVert}$ to guarantee convergence as in the standard Landweber iteration scheme~\cite{Rieder:keineProblemeMitInversenProblemen, Engl1996RegularizationOI}.
    
    For noisy measurement data $\psi^\delta$ with noise level $\delta$ fulfilling $\|\psi - \psi^{\delta}\| \leq \delta$, the iteration is stopped when either the discrepancy $\| A \vartheta_k - \psi \|$
    does not decrease further and starts increasing or when a threshold is reached, i.e.,
    \begin{align}
    \label{eq:discrepancyTimeDer}
         \| A \vartheta_k - \psi \| \leq \rho \delta,
    \end{align}
    with a fixed $\rho > 1$.
\end{algo}

\begin{remark}
    For the numerical implementation we can use the discrete cosine transform $\mathcal{C}_{d}$ to compute $\mathcal{C}_{(\omega, \tfreq )}(\vartheta_{k + \frac{1}{2}})$ for some fixed $\omega$ and $\tfreq $ in the second step in~\eqref{eq:algo_FourierCoef}. After dividing by the factor $(1 + \tau\alpha)  + \tau \beta \frac{(\tfreq ')^2}{|\omega'|^2 + \gamma}$ as in~\eqref{eq:algo_FourierCoef}, we get the solution using the inverse discrete Fourier transform by $\vartheta_{k+1} = \mathcal{C}_{d}^{-1} \left(\mathcal{C}_{d} (\vartheta_{k+1})\right).$
  %  \begin{align*}
  %      \vartheta_{k+1} &= \mathcal{C}_{d}^{-1} \left(\mathcal{C}_{d} (\vartheta_{k+1})\right).
  %  \end{align*}
\end{remark}

\section{Numerical application to dynamic computerized tomography}
\label{sec:NumApplicationDynCT}As an example we consider dynamic computerized tomography (CT). We apply the dual method~\ref{alg:dualMethod}, i.e., Tikhonov regularization in Banach spaces, as well as Algorithm~\ref{algo:timeDerivative}, which penalizes the time derivative in Hilbert spaces.  %\todo[inline]{Reference equation and repeat what's the difference between the two. Readers/Reviewer can be lazy.}

First, we introduce the mathematical model of dynamic CT in a similar way as in~\cite{Burger_2017} and compute the static and dynamic adjoint operators in the Banach space setting, which can be derived from the Hilbert space case considered, e.g., in~\cite{natterer}. In Section~\ref{sec:Testdata}, we introduce the different test cases, two simulated phantoms and one measured data set. The Sections~\ref{sec:DualMethodNumerics} and~\ref{sec:TimeDerivativeNumerics} contain different experiments for the dual method~\ref{alg:dualMethod} and Algorithm~\ref{algo:timeDerivative}, respectively, which is followed by a discussion of the results.

%\todo[inline]{I think we need an overview of the section here: What are the different test cases and what is the purpose of each experiment that will follow.}

\subsection{Dynamic computerized tomography}
%First, we repeat the mathematical model of dynamic CT which is similar to the one in~\cite{Burger_2017} and compute the static and dynamic adjoint operators in the Banach space setting which can be derived from the Hilbert space case considered, e.g. in~\cite{natterer}.

The static Radon transform $R : L^r(B) \to W$ is defined by $(R f)(\varphi, \sigma) = g(\varphi, \sigma) = \int_{-1}^1  f\left(\sigma \theta(\varphi) + w \theta(\varphi)^{\bot}\right) \dw$
%\begin{align*}
%    (R f)(\varphi, \sigma) = g(\varphi, \sigma) = \int_{-1}^1  f\left(\sigma \theta(\varphi) + w \theta(\varphi)^{\bot}\right) \dw
%\end{align*} %\todo{$\dr$ doesn't work correctly. Also note overlap with $L^r$.}
%\todo[inline]{Arguements $(Rf)(r,\theta)$?, warum integrieren wir auch $\varphi$?} 
for $1 < r < \infty$, the unit circle $B \subset \R^2$ with $\theta(\varphi) = ( \cos(\varphi), \sin(\varphi))$ and $\theta(\varphi)^{\bot} = ( -\sin(\varphi), \cos(\varphi))$. This is not a constraint since each bounded domain in $\R^2$ can be scaled to the unit circle. The space $ W \coloneqq L^s([-1,1] \times [0,\pi), (1-\sigma^2)^{-\frac{s-1}{2}})$ is the space of all functions $g$ on $[-1,1] \times [0,\pi) $ with
\begin{align}
\label{eq:normW}
	 \lVert g \rVert_{W}^s \coloneqq \int_0^\pi \int_{-1}^1 \lvert g(\varphi, \sigma) \rvert^s 
	 (1-\sigma^2)^{-\frac{s-1}{2}} \dsig \dphi < \infty,
\end{align}
%\todo[inline]{@Martin: Is the duality pairing $\langle \cdot, \cdot \rangle_{W,W^*}$ correct? Compare with Lemma 2.1 from~\cite{Burger_2017}.}
%
%The dual space $W^*$ can be identified with a space function space equipped with the norm
%\begin{align*}
%\| g \|^{m^*} &\coloneqq \int_0^\pi \int_{-1}^1 \lvert g(\varphi, \sigma) \rvert^{m^*} 
%	 \left((1-\sigma^2)^{-\frac{m-1}{2}}\right)^{1 - m^*} \dsig \dphi \\
%     &= \int_0^\pi \int_{-1}^1 \lvert g(\varphi, \sigma) \rvert^{m^*} 
%	 \sqrt{1-\sigma^2} \dsig \dphi,
%\end{align*}
since $s^* = \frac{s}{s-1}$ as in~\cite[Section 3.2]{KazimierskiSmoothnessBesov}.

\begin{lemma}
    \label{lem:staticRadon}
    If $s \leq r$, %the static operator 
    $R : L^r(B) \to W$ is linear and bounded and its adjoint $R^* : W^* \to L^{r^*}(B)$ is given by
        \begin{align*}
            (R^* g)(\spatial) = \int_0^\pi g(\spatial\cdot \theta(\varphi),\varphi) \left(1 - (\spatial \cdot \theta(\varphi))^2 \right)^{-\frac{s-1}{2}} \dphi.
        \end{align*}
\end{lemma}

The proof is analogous to the case $r=s=2$ that can be found e.g. in~\cite[Chapter 2]{natterer}.

Now, we consider the dynamic Radon transform defined by $ (\dynR  \vartheta)(t) = \psi(t) = R \vartheta(t)$
%\begin{align*}
%    (\mathcal{R} \vartheta)(t) = \psi(t) = R \vartheta(t)
%\end{align*}
for each $t \in [0,T]$.
In the following we denote the Lebesgue-Bochner spaces by $L^p(0,T;L^r(B)) =: \X$ and $ L^q(0,T;W) =: \Y$.
\begin{lemma}
    For $s \leq r$ and $q \leq p$ the dynamic operator $\dynR : \X \to \Y$ is linear and bounded and its adjoint $\dynR^*: \Y^* \to \X^*$ is given by
    \begin{align*}
        (\dynR^* \psi^*)(t) = \int_0^\pi \psi^*(t)(\spatial \cdot \theta(\varphi)) (1-(\spatial \cdot \theta(\varphi))^2)^{-\frac{s-1}{2}} \dphi . 
    \end{align*}
\end{lemma}
\begin{proof}
 As in the static case, the dynamic Radon operator $\dynR$ is clearly linear. We start by showing the boundedness of the operator. For $\vartheta \in \X$ we have
    \begin{align*}
            \lVert \dynR \vartheta \rVert_\Y 
            = \left( \int_0^T \lVert R \vartheta(t) \rVert_{W}^q \dt \right)^{\frac{1}{q}}
            \lesssim \left( \int_0^T \lVert \vartheta(t) \rVert_{L^r(B)}^q \dt \right)^{\frac{1}{q}}
            = \lVert \vartheta \rVert_{L^q(0,T;L^r(B))}
            \lesssim \lVert \vartheta \rVert_\X
    \end{align*}
    where the first inequality follows due to the boundedness of the static Radon operator~\ref{lem:staticRadon} and the last equality is an application of Lemma~\ref{lem:embeddingBochner}.
    
    Now, we compute the adjoint operator for $\vartheta \in \X$ and $\psi^* \in \Y^*$ by
    \begin{align*}
        \langle \dynR \vartheta, \psi^* \rangle_{\Y,\Y^*}
        = \int_0^T \langle R \vartheta(t), \psi^*(t) \rangle_{W,W^*} \dt
        = \int_0^T \langle \vartheta(t), R^* \psi^*(t) \rangle_{L^r(B), L^{r^*}(B)} \dt
        = \langle \vartheta, \dynR^* \psi^* \rangle_{\X,\X^*}
    \end{align*}
    using the adjoint static operator from Lemma~\ref{lem:staticRadon}.
\end{proof}
%\todo[inline]{Eindeutigkeit, sobald jeder Winkel zumindest einmal in der Zeit vorkommt?}

\subsubsection{Discretization of the operator}
\label{sec:discrOperator}

The continuous Radon transform can be discretized for different measurement geometries. Here, we consider two cases for later applications.  
%\todo[inline]{There may be more than the two we consider? Maybe remove "for two" and write, we consider here two cases. } 
In the first case, the object is scanned for each time step from the same angles, meaning that the tomograph is not rotating and we only have projections from a small set of angles. In the second case, the tomograph is rotated such that the projection angles are different in each time step, but still equidistantly spaced. In both cases, we discretize the norm of $W$, see~\eqref{eq:normW}, for $g \in \R^{N\times M}$ as
\begin{align*}
   \lVert g \rVert_{\bar{W}}^s \coloneqq
   2\pi h_\sigma h_\theta \sum_{i=1}^N \sum_{j=1}^M \lvert g_{i,j} \rvert^s (1-\sigma_i^2)^{-\frac{s-1}{2}} 
\end{align*}
for $N$ equidistantly measured offset values $\sigma_i$ discretizing the sensor with distance $h_s$ and $M$ equidistantly measured angles with distance $h_\theta$.

Since in the second case the measurement geometry changes for each time step, the indices $j$ of the data $g_{i,j}$ denote different angles. This does not make a difference for the computation of the norm, but for the computation of the forward and adjoint operator we need to take into account from which angles the object was measured at which time steps.

Of course it would also be possible to use angles that are not equidistantly spaced. However, then we would need to choose a different discretization that approximates the integral for non-equidistant support points.

In the following, we consider the second case for the simulated phantoms: If the first time-step is measured from angles $1,21,41,\ldots$, the second time-step is measured from angles $2,22,42,\ldots$ and so on which means that the object is scanned from each angle only once. In this case we have $20$ time-steps in total. The total number of angles is equidistantly spaced between $0$ and $\pi$.

For the emoji data set, we consider the first case, where the measurement angles do not change.

\subsection{Test data}
\label{sec:Testdata}
%\todo[inline]{Should we use subsubsections here to refer better?} Yes!
\subsubsection{Simulated phantom: intensity preserving} 
\label{sec:intensityPreservingPhantom}

We start by applying our methods to a simulated intensity preserving phantom with values in $[0,1]$. The phantom consists of two rectangles and a circle, which are stretched vertically by the factor $1.5$ with constant speed, see Figure~\ref{fig:groundtruthSimulated_intensity}.
With the notation of Section~\ref{paragraph:timeDerivativeSpaceH-1} %and~\eqref{eq:constructDynFunction_intensity} 
the object at time $t=0$ can be seen as the function $\statObj(\spatial)$ which is undergoing the motion $\Gamma(t,\cdot) = \Phi^{-1}(t,\cdot)$ with $\Phi(t,\spatial_1,\spatial_2) = (\spatial_1, (1+ \frac{t}{2}) \spatial_2)^{\top}$ for $t \in [0,1]$. Thus the time-dependent function $\vartheta$ is given by $\vartheta(t,\spatial) = \statObj(\Gamma(t,\spatial))$ and the velocity is $V(t,\spatial) = \partial_t \Phi(t,\spatial) = (0,\frac{\spatial_2}{2})^\top$ for $t \in (0,1)$. 

To apply Algorithm~\ref{algo:timeDerivative}, we have to verify that the time-dependent function $\vartheta$ we are considering lies in the space
$\X = L^2(0,T;L^2(\Omega)) \cap H^1(0,T;H_0^{-1}(\Omega))$.

We see that $V$ is constant in time and in the horizontal component but not in the vertical component. Clearly it holds $V \in L^\infty(0,T;L^\infty(\Omega))$ and the movement preserves intensity. Thus we get with Lemma~\ref{lem:timeDerInH-1} that $\vartheta \in \X$.
The phantom consists of $20$ equidistant time steps in $[0,1]$ with a spatial resolution of $41 \times 41$ pixels.

\begin{figure}[htbp]
    \centering
    \renewcommand{\thesubfigure}{\arabic{subfigure}}
    \setcounter{subfigure}{1}
    
    % Erste Reihe
    \begin{subfigure}{0.15\textwidth}
        \includegraphics[width=\textwidth]{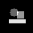}
        \caption{$t=0$}
    \end{subfigure}
    \begin{subfigure}{0.15\textwidth}
        \includegraphics[width=\textwidth]{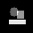}
        \caption{$t=5$}
    \end{subfigure}
    \begin{subfigure}{0.15\textwidth}
        \includegraphics[width=\textwidth]{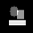}
        \caption{$t=10$}
    \end{subfigure}
    \begin{subfigure}{0.15\textwidth}
        \includegraphics[width=\textwidth]{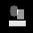}
        \caption{$t=15$}
    \end{subfigure}
    \begin{subfigure}{0.15\textwidth}
        \includegraphics[width=\textwidth]{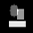}
        \caption{$t=19$}
    \end{subfigure}
\caption{time steps $0,5,10,15,19$ for the ground truth of the intensity preserving phantom}
\label{fig:groundtruthSimulated_intensity}
\end{figure}

After computing the sinogram as described in Section~\ref{sec:discrOperator} for $7$ angles per time-step and $40$ offset values equidistantly spaced in $[-1,1]$, we add Gaussian noise with a mean of $0$ and a standard deviation of $0.05$. To prevent committing an inverse crime, we compute our reconstruction with a spatial resolution of $33 \times 33$ pixels. We use a low resolution to allow running a large number of experiments for the dual method~\ref{alg:dualMethod}. The total computation time was around three days.
For comparison, we use the same resolution for our proposed method~\ref{algo:timeDerivative} which considers the time-derivative even if the computation time is faster.

\subsubsection{Simulated phantom: mass preserving} 
\label{sec:massPreservingPhantom}

The second phantom we are considering preserves mass. It consists of a static circle and a moving rectangle that changes size and intensity, see Figure~\ref{fig:groundtruthSimulated_mass}. The phantom consists of 20 equidistant time steps in $[0,1]$ with a spatial resolution of $83 \times 83$ pixels. After computing the sinogram as described above for 7 angles per time-step and $160$ offset values equidistantly spaced in $[-1,1]$, we add Gaussian noise with a mean of $0$ and a standard deviation of $0.05$. To prevent committing an inverse crime, we compute our reconstruction with a spatial resolution of $61 \times 61$ pixels. Here, we choose a higher resolution since the phantom is indistinguishable in the small resolution we used for the intensity preserving phantom above. Since we only apply method~\ref{algo:timeDerivative} to this phantom, the higher resolution poses no problem.
% \todo{Why are the resolutions different to above? G: I first did the intensity preserving on a smaller resolution. That's also good for the Tikhononvreconstruction since it is so slow. The problem ist, that the mass preserving phantom needs a higher resolution, since the rectangle is rotating. Thus, I used a higher resolution but stayed with the smaller one for the intensity preserving case for comparison with Tikhonov. Should I change this?}

\begin{figure}[htbp]
    \centering
    \renewcommand{\thesubfigure}{\arabic{subfigure}}
    \setcounter{subfigure}{1}
    
    % Erste Reihe
    \begin{subfigure}{0.15\textwidth}
        \includegraphics[width=\textwidth]{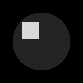}
        \caption{$t=0$}
    \end{subfigure}
    \begin{subfigure}{0.15\textwidth}
        \includegraphics[width=\textwidth]{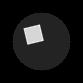}
        \caption{$t=5$}
    \end{subfigure}
    \begin{subfigure}{0.15\textwidth}
        \includegraphics[width=\textwidth]{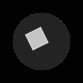}
        \caption{$t=10$}
    \end{subfigure}
    \begin{subfigure}{0.15\textwidth}
        \includegraphics[width=\textwidth]{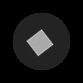}
        \caption{$t=15$}
    \end{subfigure}
    \begin{subfigure}{0.15\textwidth}
        \includegraphics[width=\textwidth]{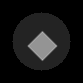}
        \caption{$t=19$}
    \end{subfigure}
\caption{time steps $0,5,10,15,19$ for the ground truth of the mass preserving phantom}
\label{fig:groundtruthSimulated_mass}
\end{figure}

\vspace{-3ex}
\subsubsection{Emoji data set}
\label{sec:emojiData}

Next, we apply our method to measured data consisting of an emoji evolving over time~\cite{EmojiDataset, EmojiArticle}. To use the measured data for our method, we first convert the fan-beam sinograms to parallel-beam sinograms using the inbuilt MATLAB rebinning algorithm \texttt{fan2para} where $D$ is the distance from the fan-beam vertex to the center of rotation (origin) in image pixels. This can be chosen as $D = D_\mathrm{fo}/d_\mathrm{eff}$, where $D_\mathrm{fo}$ is the focus-to-origin distance and $d_\mathrm{eff} = d/M$ is the effective pixel size, taking into account the geometric magnification $M$ of the fan-beam and the detector pixel size $d$. The magnification is given by $M = D_\mathrm{fd}/ D_\mathrm{fo}$, where $D_\mathrm{fd}$ is the focus-to-detector distance. From~\cite{EmojiArticle} we know that $D_\mathrm{fo}$ is $540$ mm and $D_\mathrm{fd}$ is $630$ mm. The size of one physical detector pixel in the custom-made measurement device at the University of Helsinki is $0.05$ mm. We assume that first some of the pixels have been cut off in order to compensate for a misaligned center of rotation, and then the resulting sinogram was binned with factor 10, resulting in a detector pixel size $d$ of $0.5$ mm in the emoji data. This results in an effective pixel size of $\frac{D_\mathrm{fo}}{ D_\mathrm{fd}} \cdot 0.5 \: \mathrm{mm} = \frac{540}{630} \cdot 0.5 \: \mathrm{mm} \approx 0.429$ mm. Note that although $D_\mathrm{fo}$ technically cancels out when computing $D$, it must still be known to attach a physical meaning to the pixel size in the reconstruction.

%Next, we apply our method to measured data consisting of an emoji evolving over time~\cite{EmojiDataset, EmojiArticle}. To use the measured data for our method, we first convert the fan-beam sinograms to parallel beam sinograms using the inbuild matlab binning algorithm 'fan2para' where $D$ the distance from the fan-beam vertex to the center of rotation in pixels is chosen as $D = 540/0.429$. From~\cite{EmojiArticle} we know that the focus-to-object distance is $540$mm which we need to convert into pixels by dividing by the effective pixel size. The size of one pixel in the custom-made measurement device from the university of Helsinki is usually $0.05$ mm. We assume that first some of the pixels have been cut off and then the resulting sinogram was binned with factor 10, resulting in a pixel size of $0.5$ mm. This results in an effective pixel size of $\frac{\text{focus-to-object distance}}{\text{focus-to-detector distance}} \cdot 0.5 = \frac{540}{630} \cdot 0.5 \approx 0.429$.

 For each of the $15$ different emoji faces we get one time step. For each time-step we use measurements from the same $5$ equidistantly spaced angles. We estimate the noise of the measured data by averaging over a part of the sinogram with very small values. To get values that are smaller than $1$ in the reconstruction, we scale the sinogram by $1/7.5$.

\begin{figure}[H]
\vspace{-1ex}
\centering
    \includegraphics[width=0.68\textwidth]{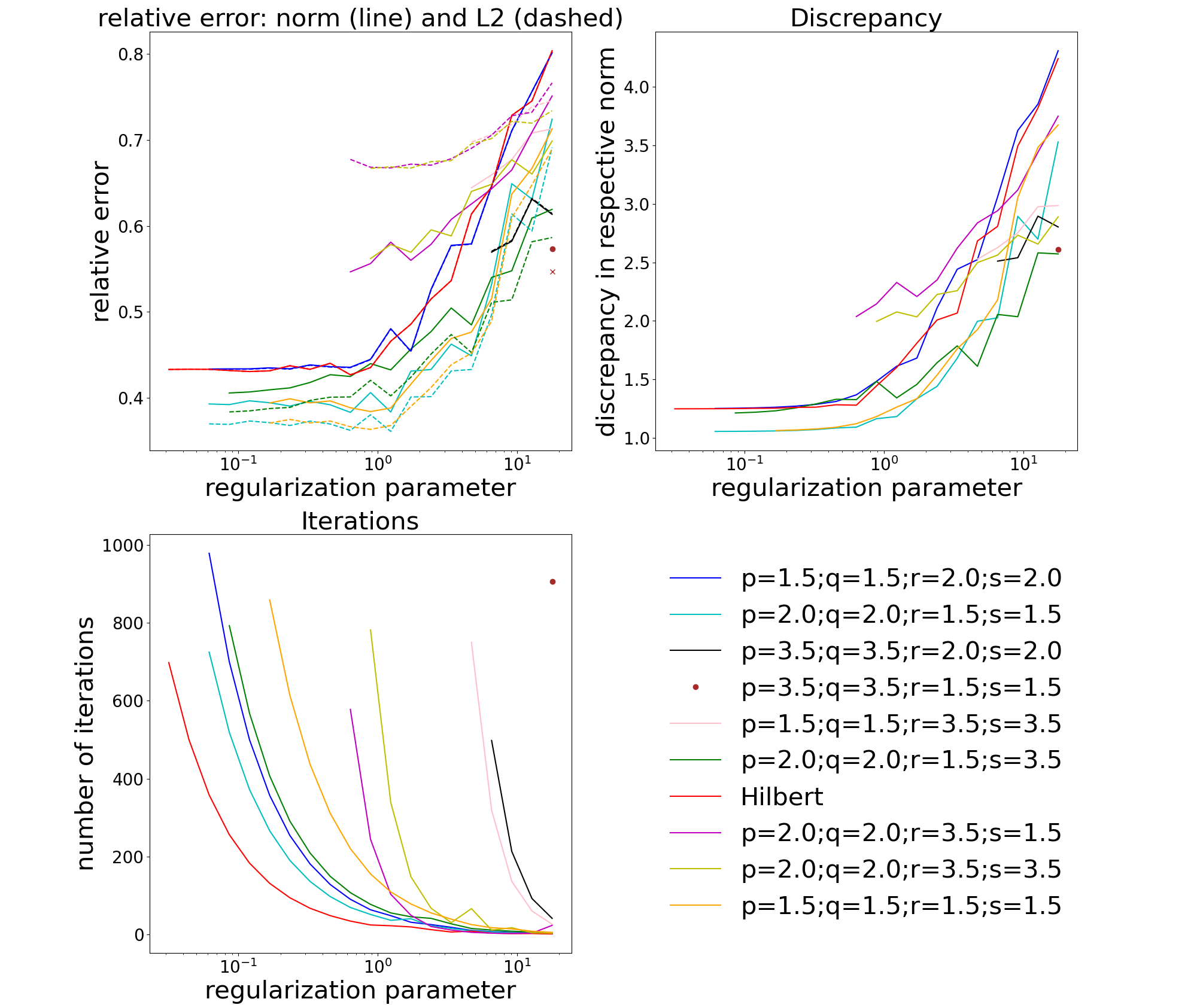}
    \caption{ \textbf{Experiment description:} Tikhonov regularization for different dynamic Banach space settings. The relative error w.r.t.~the groundtruth in the respective norm and in the $L^2$-norm, the discrepancy and the number of iterations are displayed depending on the regularization parameter. 1000 iterations take approximately 30 minutes computing time on a regular laptop. \\
    The iteration is stopped as described in Section~\ref{sec:dynamicSetting}. Only those results are displayed where the iteration is stopped before 1000 iteration. 
    In the setting $p=q=3.5$ and $r=s=1.5$ the algorithm converges only with the last regularization parameter in less than 1000 iterations. This is indicated with the dot for the error in the respective norm, the discrepancy and the iterations. The cross indicates the $L^2$-error.
  }
    \label{fig:BanachPlots}
\end{figure}

\subsection{Tikhonov regularization with power-type penalties in Banach spaces}
\label{sec:DualMethodNumerics}

First, we consider Tikhonov regularization in Banach spaces where we compare two solution approaches and apply them to the intensity preserving phantom, see Section~\ref{sec:intensityPreservingPhantom}: First, solving the dynamic minimization problem in Lebesgue-Bochner spaces, which we will refer to as 'dynamic setting', and second, solving the minimization problems in Lebesgue spaces for each measured time-point separately, which we will refer to as 'static setting'. In both cases we do not assume any knowledge about the motion.

By comparing these two approaches, we want to examine how the inclusion of time impacts the regularized solution.

We aim at minimizing the Tikhonov functional 
\begin{align*}
    T_\alpha(\vartheta) \coloneqq \frac{1}{u} \lVert A \vartheta - \psi \rVert_\Y^u + \alpha \frac{1}{v} \lVert \vartheta \rVert_\X^v .
\end{align*}
Since the inverse problem of (dynamic) computerized tomography is linear, the minimizers of the Tikhonov functional are well-defined and thus provide a regularization method, see Section~\ref{sec:TikhonovDual}.

For noisy measurement data $\psi^\delta$ with noise level $\delta$, such that $\|\psi - \psi^{\delta}\| \leq \delta$, we cannot expect that the stopping criterion $0 \in \partial T_\alpha(\vartheta_n)$ in Algorithm~\ref{alg:dualMethod} for exact data will be fulfilled. Therefore we need a different stopping criterion. Here, the iterations are stopped when either the data discrepancy $\| A \vartheta_k - \psi \|$
    does not decrease further and starts increasing or when a fixed tolerance is reached, i.e.,
    \begin{align}
    \label{eq:discrepancyTikhonov}
         \| A \vartheta_k - \psi \| \leq \rho \delta,
    \end{align}
    with a fixed $\rho > 1$. In this section, we chose $\rho = 1.05$. Furthermore, for computational reasons we only consider reconstructions that stop with these criterion after a maximum of 1000 iterations.

\subsubsection{Dynamic setting}
\label{sec:dynamicSetting}
In the dynamic setting we choose
\begin{align}
    \label{eq:functionSpacesTikhonov}
    \X = L^p(0,T;L^r(B)), \quad \Y = L^q\left(0,T;L^s\left([-1,1] \times [0, \pi), (1-\sigma^2)^{-\frac{s-1}{2}}\right) \right)
\end{align}
and execute the algorithm for different choices of the exponents $p,q,r,s>1$.
In accordance with the convergence rates from Section~\ref{sec:TikhonovDual}, we choose 
  $  v = \max(2,p,r)$ and $u = \max(2,q,s)$.

Each Banach space setting is considered for different regularization parameters. %The iteration is stopped as soon as the data discrepancy gets larger if the discrepancy in the corresponding norm reaches 1.05 \todo{why 1.05?} or if 1000 iterations are reached. 
Figure~\ref{fig:BanachReco} shows, for each setting, %\todo{order of Figure 3 before 4} 
a reconstruction obtained using the regularization parameter that yields the smallest discrepancy. 
In Figure~\ref{fig:BanachPlots} the number of iterations, the error and the discrepancy in the respective norm depending on the regularization parameter for the different Banach space settings are displayed.

\subsubsection{Static setting}
\label{sec:staticSetting}
In the static setting, we reconstruct each time-step separately and choose the spaces
\begin{align}
    \label{eq:functionSpacesTikhonov_static}
    \X = L^r(B), \quad \Y = L^s\left([-1,1] \times [0, \pi), (1-\sigma^2)^{-\frac{s-1}{2}}\right)
\end{align}
for different choices of $r$ and $s$. 
As above, we choose 
   $ v = \max(2,r)$ and $ u = \max(2,s)$ in accordance with the convergence rates from Section~\ref{sec:TikhonovDual}.
   
We chose the same regularization parameter for all time steps. %The iteration is stopped as in the dynamic case. 
In Figure~\ref{fig:Banach_static} a reconstruction for each setting can be seen where the regularization parameter that yields the smallest discrepancy is chosen. Here we compute the dynamic discrepancy, i.e., the discrepancy for the dynamic operator in the Lebesgue-Bochner $L^2$-norm. Additionally, the dynamic error, dynamic discrepancy and the mean of the number of static iterations are shown depending on the regularization parameters.

\subsection{Temporal variational regularization}
\label{sec:TimeDerivativeNumerics}

Next, we apply %\todo{specify} 
Algorithm~\ref{algo:timeDerivative} that penalizes the time derivative to the two simulated phantoms from Sections~\ref{sec:intensityPreservingPhantom} and~\ref{sec:massPreservingPhantom} and the emoji data set from Section~\ref{sec:emojiData}. 
For the discretization of the cosine transform we use the fast discrete cosine transform (DCT).

\subsubsection{Intensity preserving phantom} We chose the best regularization parameters $\alpha$ and $\beta$ by doing a grid search and comparing the $L^2$-error. For each set of parameters, the iteration is stopped as soon as the discrepancy gets larger.
However, we note that the error is lowest when choosing $\alpha = 0$. In Figure~\ref{fig:intensity}, the minimal reached discrepancy and error for different choices of $\beta$ and three different $H_0^{-1}$-norms~\eqref{eq:H-1GammaNorm} determined by $\gamma \in \{0.1,1,10\}$ are displayed.
Table~\ref{tab:FourierResult_intensity} and Figure~\ref{fig:intensity} show error, discrepancy, and iteration counts for reconstructions with the best regularization parameter for each $\gamma$, compared to FBP and Landweber, i.e. $\alpha = \beta = 0$.
%The error, discrepancy and number of iterations resulting from reconstructions with the best regularization parameter for each $\gamma$ are shown in Table~\ref{tab:FourierResult_intensity} and Figure~\ref{fig:intensity} %\todo{ordering of figures} 
%in comparison to the FBP and Landweber reconstruction, i.e. $\alpha = \beta = 0$. We see that 
Penalizing the time derivative leads to a reduction of noise compared to the Landweber iteration. However, the different choices of the $H^{-1}_0$-norm do not make a large impact on the image quality since the best regularization parameter is similar in all three cases. For $\gamma = 10$ the reconstruction error is slightly lower with a smaller regularization parameter.

\begin{figure}[H]
\centering
    \includegraphics[width=0.75\textwidth]{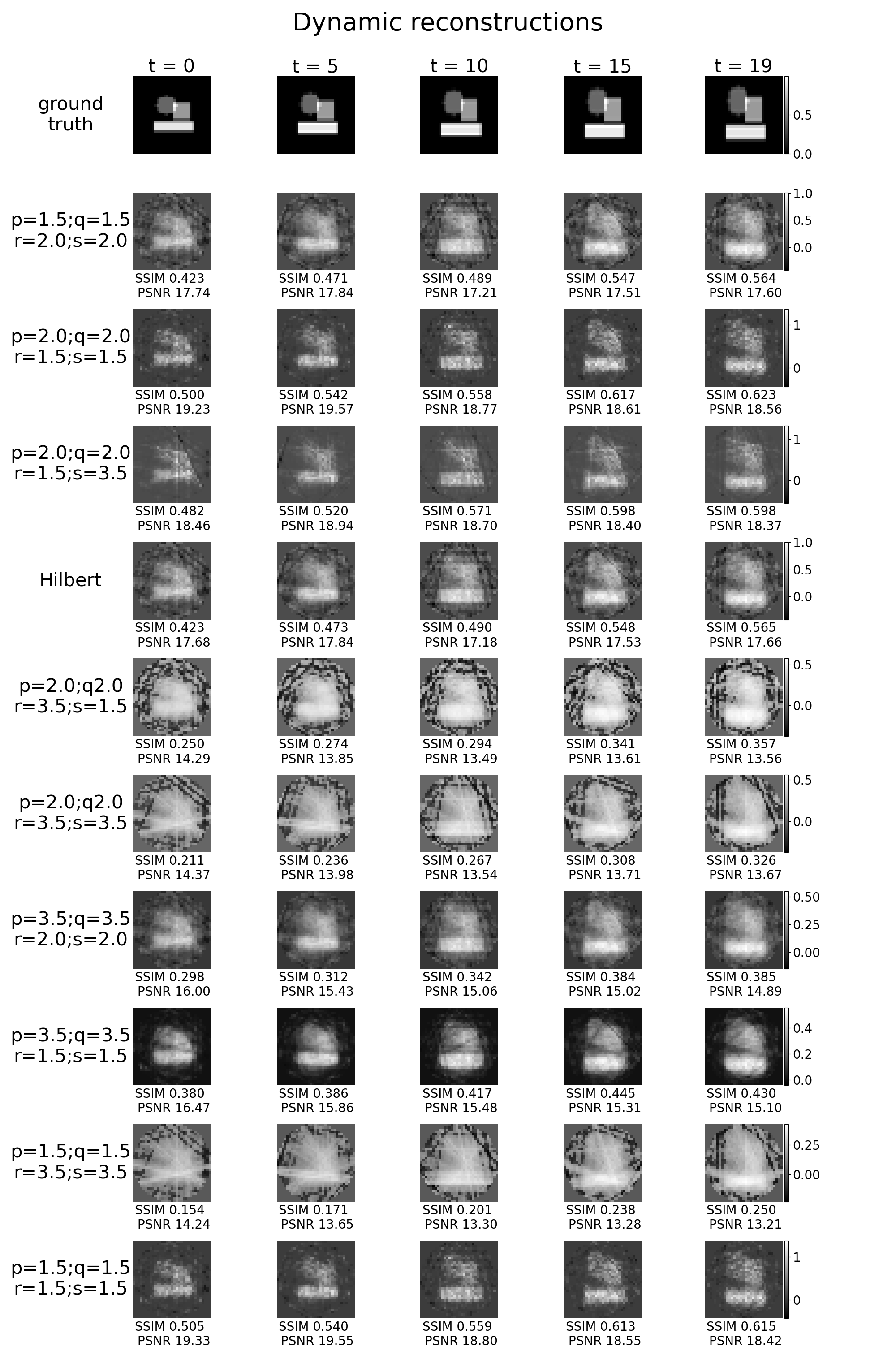}
    \caption{ \textbf{EXPERIMENT DESCRIPTION:} Tikhonov regularization for different dynamic Banach space settings. The regularization parameter is chosen such that for each setting the discrepancy is minimized, see also Figure~\ref{fig:BanachPlots}. }
    \label{fig:BanachReco}
\end{figure}

\begin{figure}[H]
\centering
\vspace{-2ex}
    \begin{subfigure}[b]{0.7\textwidth}
        \centering
        \includegraphics[width=\textwidth]{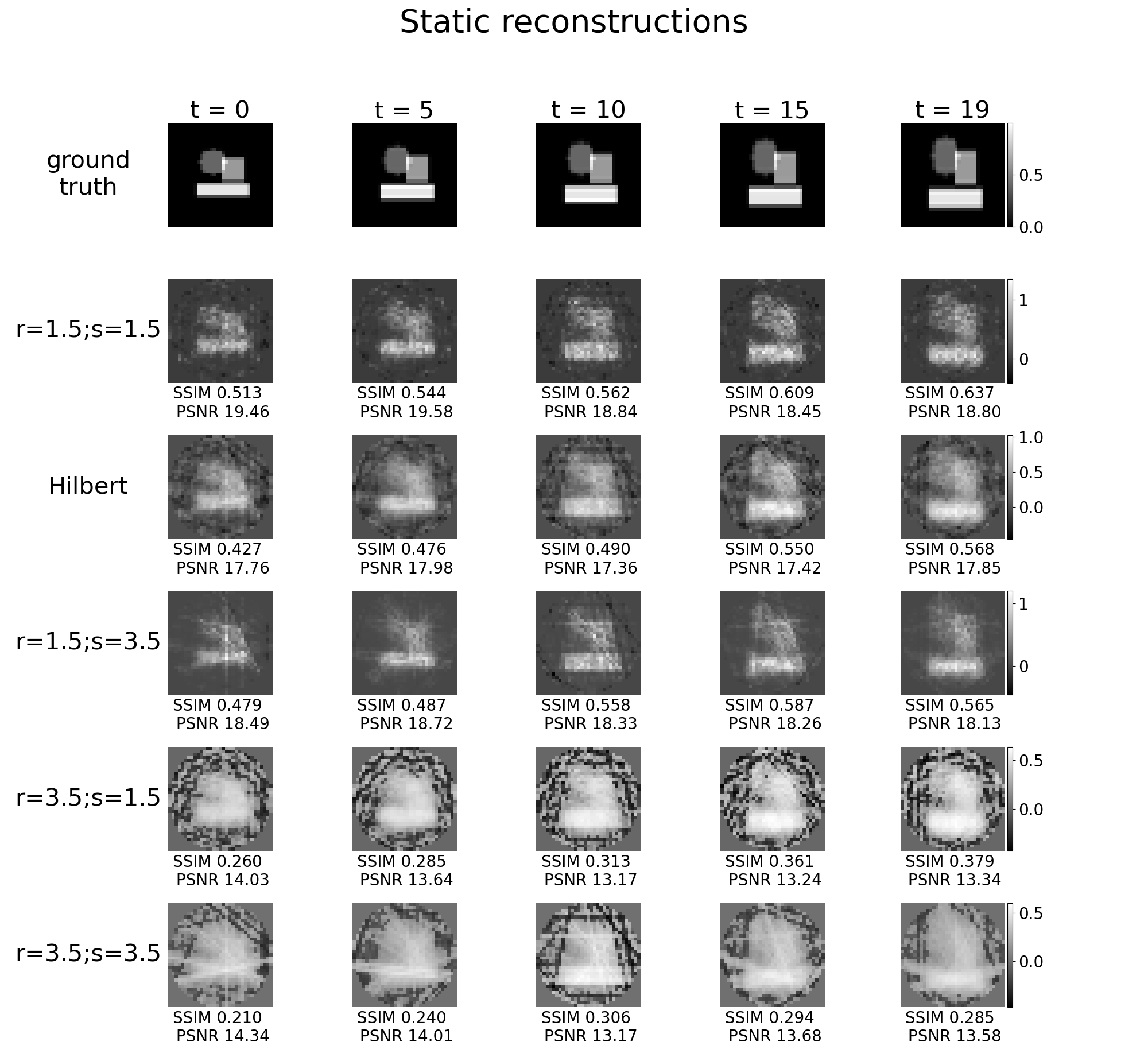}
        %\caption{}
        \label{fig:BanachReco_static}
    \end{subfigure}
    
    \vspace{-2ex} % Optionaler vertikaler Abstand

    \begin{subfigure}[b]{0.65\textwidth}
        \centering
        \includegraphics[width=\textwidth]{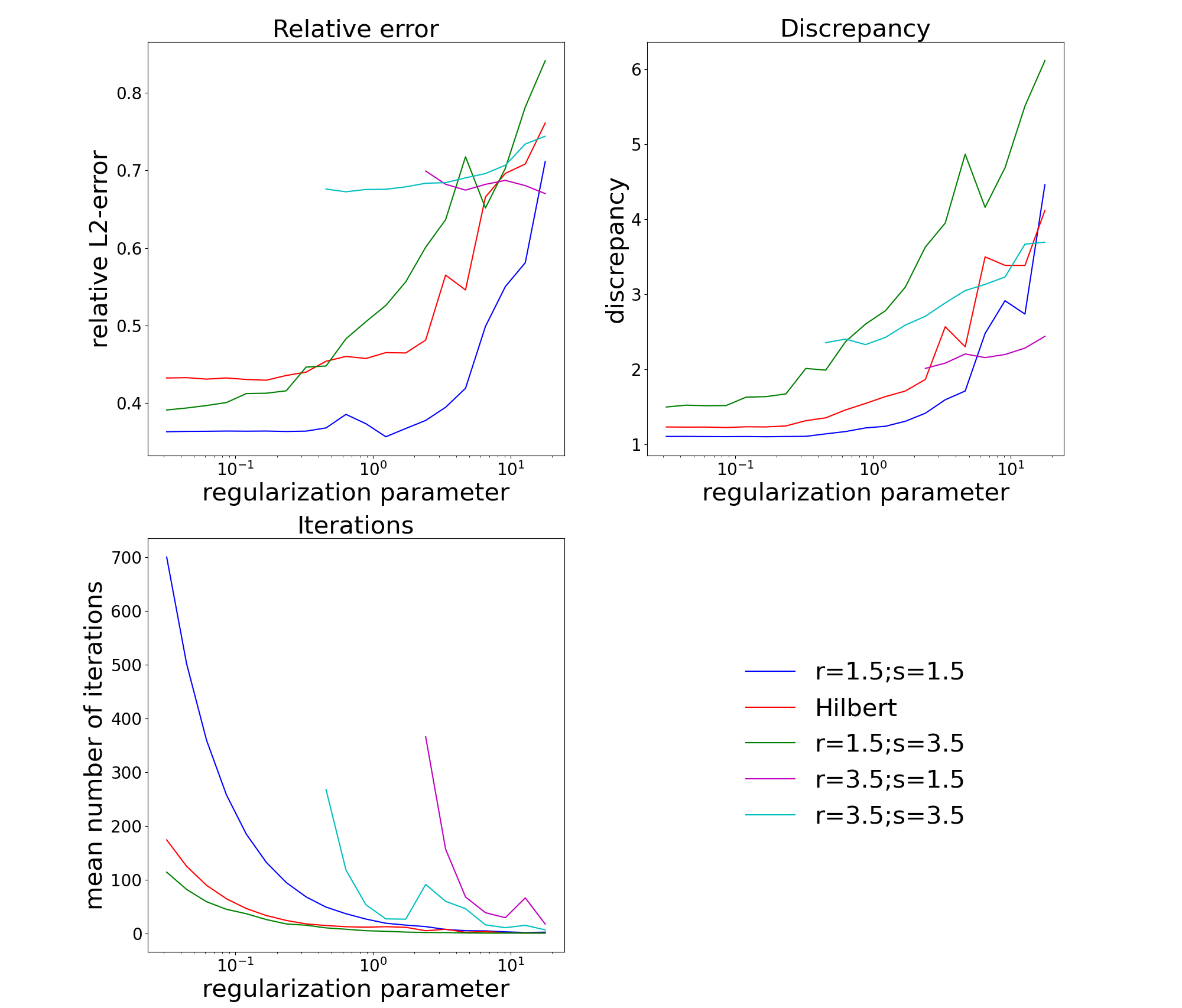}
        %\caption{}
        \label{fig:BanachPlots_static}
    \end{subfigure}
    \vspace{-2ex}
    \caption{\textbf{Experiment description:} Tikhonov regularization for different static Banach space settings. The regularization parameter has the same value for all time steps and is chosen such that for each setting the discrepancy is minimized. Reconstructions for different settings (top); Relative error in $L^2$-norm, discrepancy, and mean number of iterations depending on the regularization parameter (bottom). Results are shown only if the iteration stops before reaching 1000 steps which take approximately 30 minutes computing time on a regular laptop, see Section~\ref{sec:staticSetting} for stopping criteria.}
    \label{fig:Banach_static}
\end{figure}

\begin{figure}[H]
\centering
\vspace{-1ex}
    \begin{subfigure}[b]{0.85\textwidth}
        \centering
        \includegraphics[width=\textwidth]{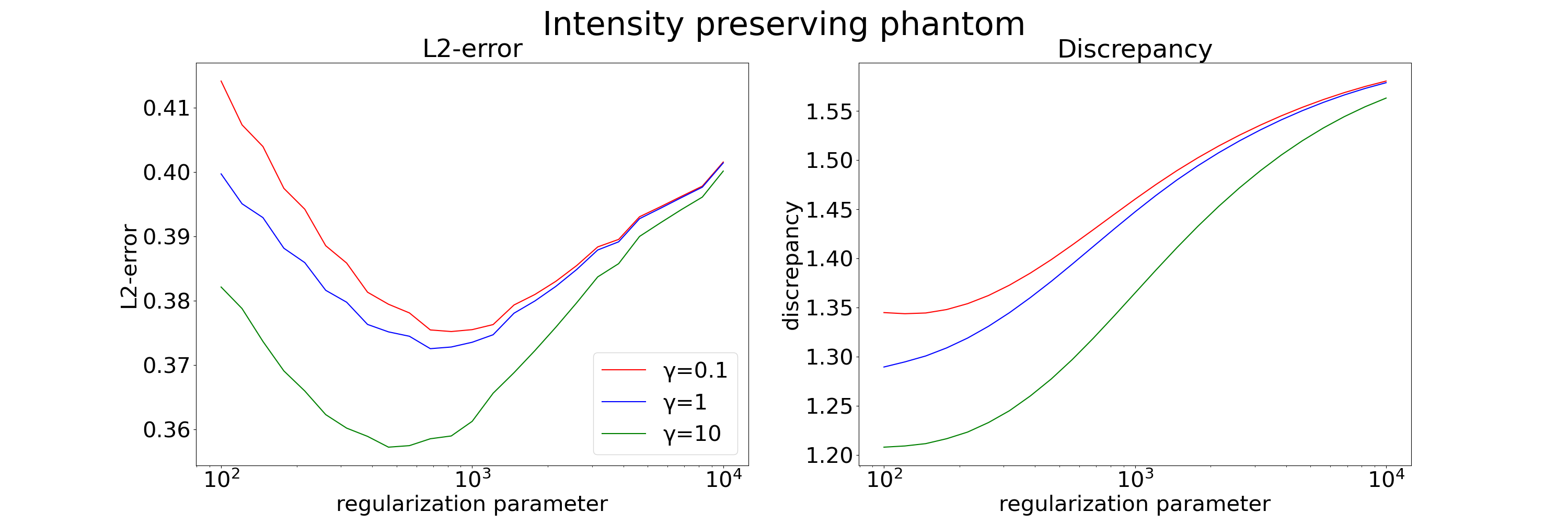}
       % \caption{}
    \end{subfigure}
    
   % \vspace{-1ex} % Optionaler vertikaler Abstand

    \begin{subfigure}[b]{0.7\textwidth}
        \centering
    \includegraphics[width=\textwidth]{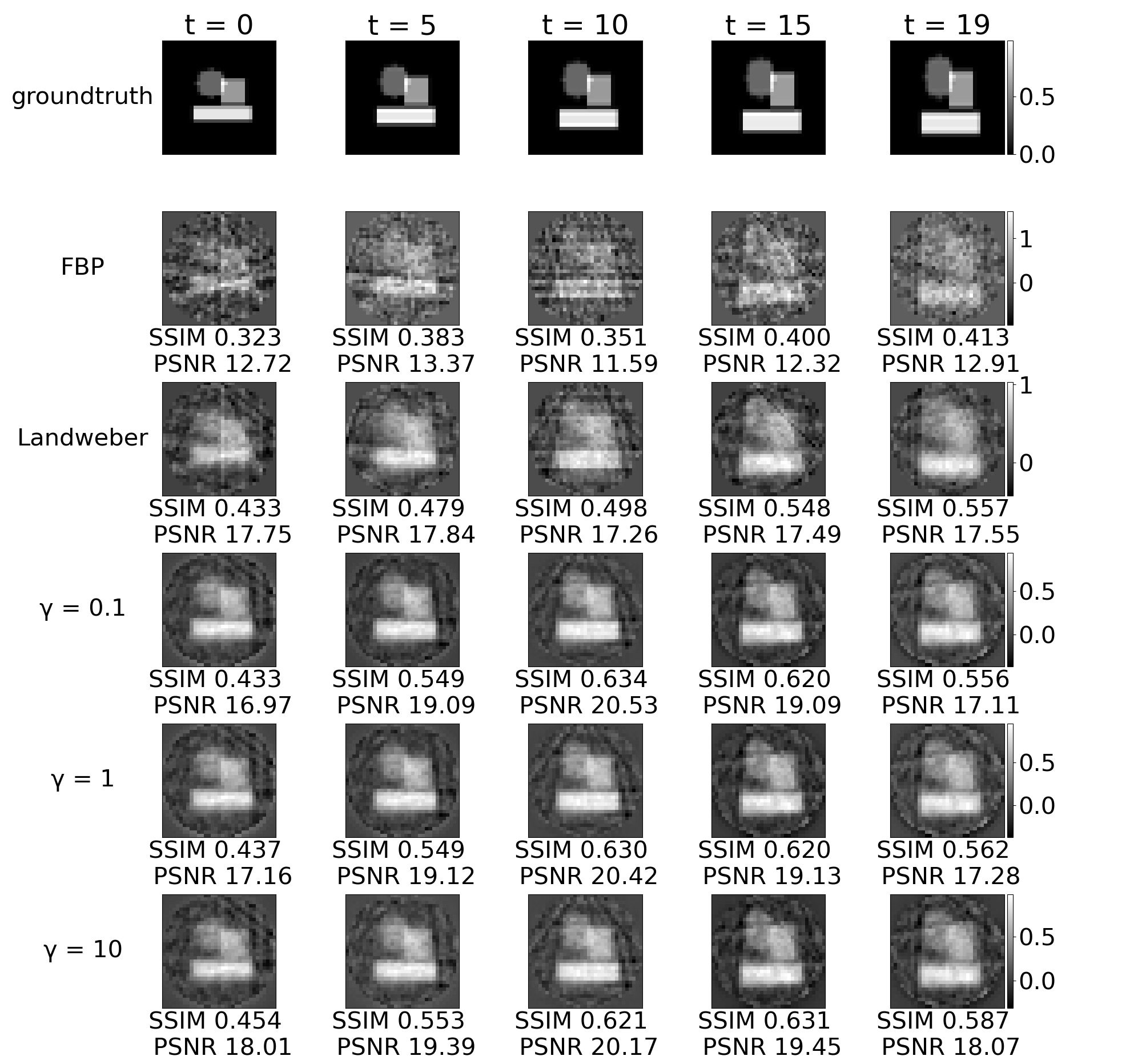}
       % \caption{}
    \end{subfigure}
    \vspace{-1ex}
    \caption{\textbf{Experiment description:} \textbf{Top:} Minimal reached discrepancy (right) and corresponding error (left) for different choices of regularization parameters $\beta$ and $H^{-1}_0$-norm determined by $\gamma$ for the intensity preserving phantom. \textbf{Bottom:}     Reconstructions for the intensity preserving phantom using FBP, Landweber and  Algorithm~\ref{algo:timeDerivative} with
     $\gamma = 0.1$ and $\beta = 825$,
      $\gamma = 1$ and $\beta = 825$ and
      $\gamma = 10$ and $\beta = 681$,
    at time points $t = 0,5,10,15,19$ for $7$ equidistantly spaced rotating angles per time step.}
    \label{fig:intensity}
\end{figure}

\vspace{-2ex}

\begin{table}[h]
    \centering
    \begin{tabular}{|c|c|c|c|c|c|}
        \hline
         &  FBP & Landweber  & $\gamma = 0.1$ & $\gamma = 1$ &  $\bm{\gamma = 10}$ \\
         \hline

        $L^2$-error & $75.63 \%$ & $ 42.71 \%$ & $ 37.52 \%$ & $37.25 \%$ &  $\bm{35.72 \%}$\\
        \hline
        mean SSIM & 0.3852 & 0.5068 & 0.5738 & 0.5739 & $\bm{0.5803} $\\
        \hline
        mean PSNR & 12.85 & 17.66 & 19.02 & 19.07 & $\bm{19.37}$ \\
        \hline
        discrepancy & $2.497$  & $1.265$ & $ 1.445 $ & $1.412$ & $1.277$\\
        \hline
        iterations &  -- & $11$ & $ 21 $ & $ 21 $ & $19$\\
        \hline
      %  $\alpha$ & -- & 0 & $0$ & $0$  &  0 \\
      %  \hline
        $\beta$ & -- & 0 & $825.4$ & $681.3$ &  $464.2$\\
        \hline
    \end{tabular}
    \caption{Error and discrepancy for the different methods applied to the intensity preserving phantom. For each $\gamma$ the regularization parameter with the smallest error was chosen. The regularization method with the best results is highlighted in bold. $\alpha$ is always 0.}
    \label{tab:FourierResult_intensity}
\end{table}

\subsubsection{Mass preserving phantom} For the mass preserving phantom, we stop the iteration according to the discrepancy principle~\eqref{eq:discrepancyTimeDer} with $\rho = 1.1$ or if the discrepancy starts to get larger. Since this threshold is reached for several regularization parameters, the graph describing the discrepancy in Figure~\ref{fig:mass} is non-smooth and we do not state the values of the discrepancy in Table~\ref{tab:FourierResult_mass}. As in the intensity preserving case, the spatial regularization parameter $\alpha$ has no positive effect on the reconstruction. The reconstructions can be seen in Figure~\ref{fig:mass}.

\begin{comment}\begin{figure}[htbp]
    \centering
    \includegraphics[width=0.9\textwidth]{figures/TimeDerivative/mass_plot.png}
    \caption{$L^2$-error and discrepancy for different choices of regularization parameters $\beta$ and $H^{-1}_0$-norm determined by $\gamma$ for the mass preserving phantom.}
    \label{fig:errorDiscrepancy_mass}
\end{figure}
\end{comment}

\begin{table}[h]
    \centering
    \begin{tabular}{|c|c|c|c|c|c|}
        \hline
         &  FBP & Landweber  & $\gamma = 0.1$ & $\gamma = 1$ &  $\bm{\gamma = 10}$ \\
         \hline

        $L^2$-error & $365.86 \%$ & $ 42.78 \%$ & $ 41.87 \%$ & $ 41.8 \%$ &  $\bm{41.39 \%}$\\
        \hline
        mean SSIM & 0.0426 & 0.3701 & 0.3859 & 0.3881 & $\bm{0.4064} $\\
        \hline
        mean PSNR & 3.405 & 22.08 & 22.27 & 22.29 & $\bm{22.38}$ \\
        \hline
        iterations &  -- & $6$ & $ 7 $ & $ 7 $ & $7$\\
        \hline
      %  $\alpha$ & -- & 0 & $0$ & $0$  &  $0$ \\
       % \hline
        $\beta$ & -- & 0 & $133.4$ & $133.4$ &  $177.8$\\
        \hline
    \end{tabular}
    \caption{$L^2$-error for the different methods applied to the mass preserving phantom. For each $\gamma$ the regularization parameter with the smallest error was chosen. The regularization method with the best results is highlighted in bold. $\alpha$ is always 0.}
    \label{tab:FourierResult_mass}
\end{table}

\vspace{-2ex}
\subsubsection{Emoji data}
Finally, we apply our method to the emoji data set. Since the reconstruction contained many negative values, we added a non-negativity constraint to the reconstruction algorithm~\ref{algo:timeDerivative}, i.e., after step~\eqref{eq:algo_endStep} we set negative values to $0$. The results are shown in Figure~\ref{fig:CTrecoEmoji} where we observe a similar reconstruction quality of Landweber iteration and Algorithm~\ref{algo:timeDerivative}.

\begin{comment}
\begin{figure}[htbp]
    \centering

    \includegraphics[width=0.75\textwidth]{figures/TimeDerivative/mass_reconstructions.png}
    \caption{\textbf{EXPERIMENT DESCRIPTION:}
    Reconstructions using FBP, Landweber and Algorithm~\ref{algo:timeDerivative} with $\gamma = 10$ and $\beta = 133$ for the mass preserving phantom at time points $t = 0,5,10,15,19$ for $7$ equidistantly spaced rotating angles per time step.}
    \label{fig:CTrecoTimeDerivative_mass}
\end{figure}
\end{comment}

\subsection{Discussion of results}

For the dual method~\ref{alg:dualMethod} in the static setting, i.e. reconstructing each time-step separately, we observe sparse but noisy reconstructions for $\X = L^r(B)$ with $r < 2$. The choice of the exponent for the space $\Y$ seems to have less impact on the reconstruction quality. For $r > 2$ we observe smoother reconstructions. However, here we used higher regularization parameters since the method does not converge in reasonable time (1000 iterations). 

The reconstructions obtained in the dynamic setting look quite similar, but we get more variety by choosing different exponents for time and space. We observe that the method does not converge for smaller regularization parameters if a value larger than 2 is used in the space $\X = L^p(0,T;L^r(B))$ for $p$ or $r$. Even for smaller exponents the method takes longer to converge than in the static setting. A reason for this could be the different computation of the data discrepancy. In the static setting, the algorithm is stopped as soon as the data discrepancy for a single time point is not getting smaller. In the dynamic setting, however, the iteration is only stopped when the data discrepancy for all time points computed in the Lebesgue-Bochner norm is not getting smaller. This explains the different behavior even for Hilbert spaces which leads to the same minimization problem.

We further observe that it is hard to compare the different reconstructions in quality since they are so different. The dynamic reconstruction with $p=q=3.5$ and $r=s=1.5$ uses a high exponent for the regularity in time and a small one for the spatial regularity. We observe little noise which might result from the high regularity in time and some sparsity in the spatial domain. If we compare this with the reconstruction for $p=q=1.5$ and $r=s=1.5$ where we incorporate sparsity both in time and space, we observe more noise but only slightly sparser reconstructions in time.

For Algorithm~\ref{algo:timeDerivative} we observe a different level of improvement for the intensity- and mass-preserving phantoms. While in the intensity preserving case, penalizing the time derivative improves the Landweber reconstruction noticeably, in the mass preserving case we only observe a small improvement in the relative error. The emoji data set is also mass preserving. Here we do not know the error due to a missing ground truth and do not observe noticeable improvements in the reconstruction compared to Landweber.

Compared to Tikhonov regularization, Algorithm~\ref{algo:timeDerivative} is much faster and yields better results. This is not surprising since Hilbert space algorithms are faster in each time step and we use a regularization term that incorporates the time-dependence of the phantom.

\begin{figure}[H]
\centering
\vspace{-2ex}
    \begin{subfigure}[b]{0.85\textwidth}
        \centering
        \includegraphics[width=\textwidth]{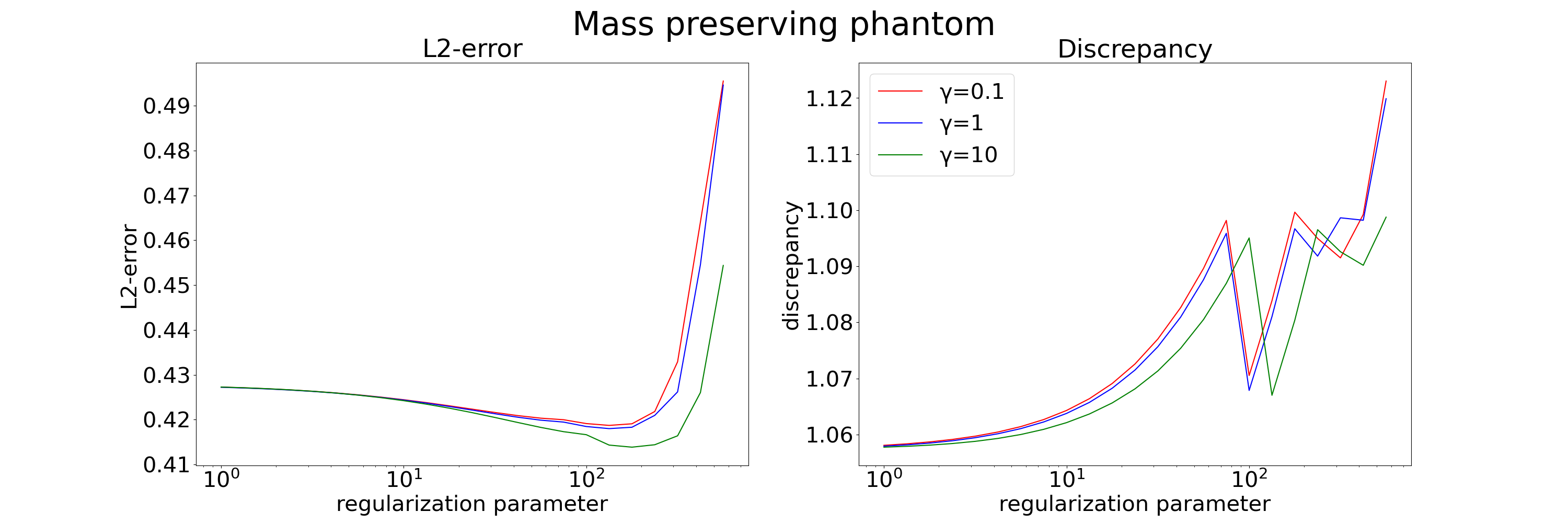}
      %  \caption{}
    \end{subfigure}
    
   % \vspace{-1ex} % Optionaler vertikaler Abstand

    \begin{subfigure}[b]{0.7\textwidth}
        \centering
    \includegraphics[width=\textwidth]{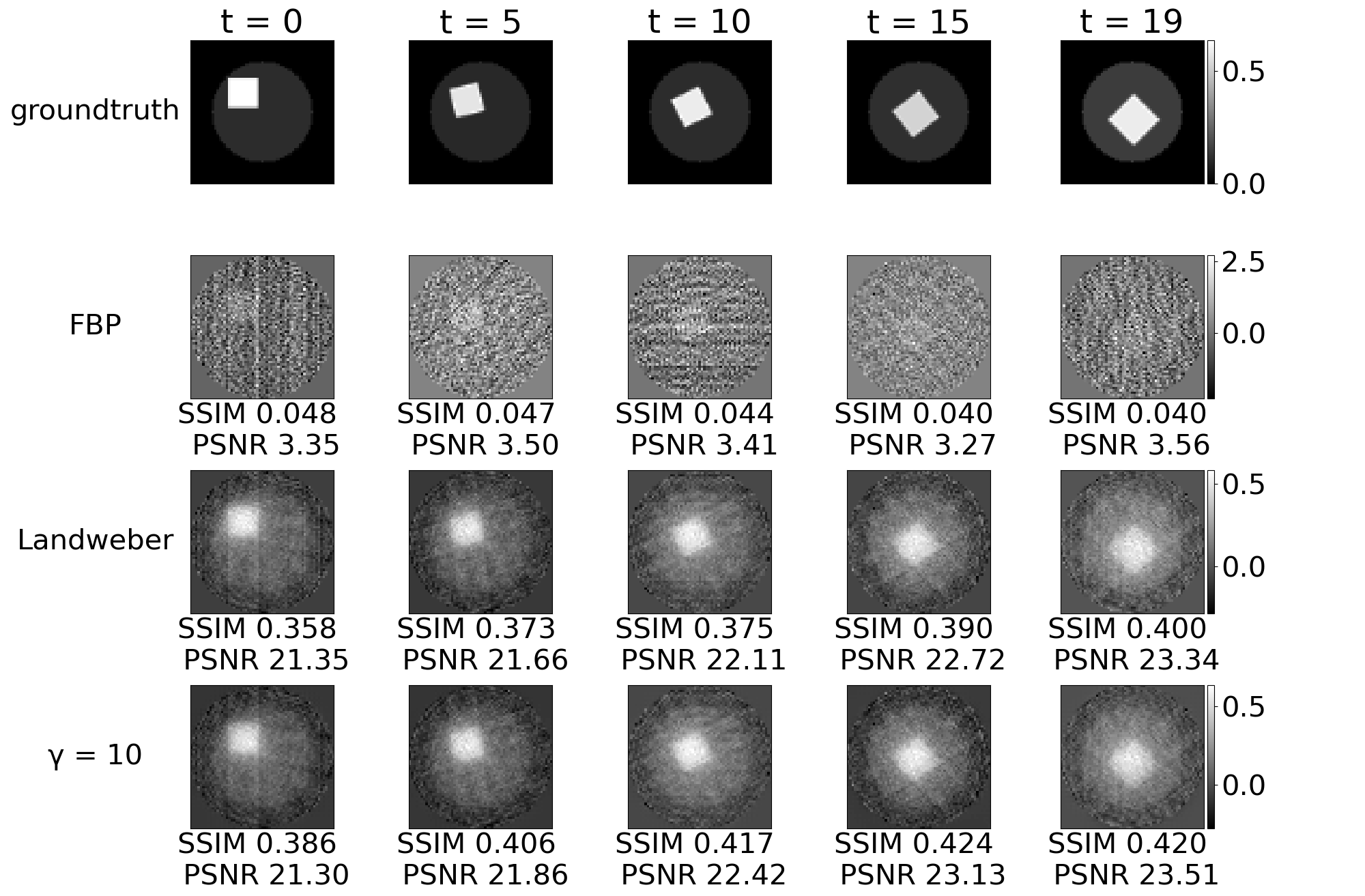}
      %  \caption{}
    \end{subfigure}
    \vspace{-1ex}
    \caption{\textbf{Experiment description:} Temporal variational regularization for the mass preserving phantom for $7$ equidistantly spaced rotating angles per time step. \textbf{Top:} $L^2$-error and discrepancy for different choices of regularization parameters $\beta$ and $H^{-1}_0$-norm determined by $\gamma$. \textbf{Bottom:} Reconstructions using FBP, Landweber and Algorithm~\ref{algo:timeDerivative} with $\gamma = 10$ and $\beta = 133$ at time points $t = 0,5,10,15,19$.}
    \label{fig:mass}
\end{figure}

\section{Conclusion and outlook}
In this article we show geometric properties of Lebesgue-Bochner spaces and discuss two possibilities to include Lebesgue-Bochner spaces in the function space setting of regularization methods for time-dependent inverse problems. First, we adapt Tikhonov regularization in Banach spaces to Lebesgue-Bochner spaces and second, we solve a variational problem penalizing the time derivative. To support our findings, we apply these algorithms to dynamic computerized tomography which serves as an example problem.

Further research objectives include the analysis of different stopping criteria to obtain convergence for both algorithms and investigations of choices of the exponents of the Lebesgue-Bochner space in the dual method~\ref{alg:dualMethod} solving the Tikhonov functional. Furthermore, the methods could be extended to nonlinear or parameter identification problems.

By combining the penalization of the time derivative with the abstract framework for parabolic parameter identification problems~\cite{kaltenbacher17}, we could extend our method~\ref{algo:timeDerivative} to this class of problems and Sobolev-Bochner spaces, a natural extension of classical Sobolev spaces to time-dependent functions.

\paragraph{Data availability statement}
The code and data that support the findings of this article are openly available on GRO.data~\cite{GRODATA} at \url{https://doi.org/10.25625/0RNNXC}.

\vspace{-1ex}
\paragraph{Acknowledgements}
The authors thank T.~Heikkilä, A.~Meaney, T.~T.~N.~Nguyen, S.~Siltanen and B.~Sprung for helpful discussions.

G.S., T.H. and A.W. acknowledge funding by Deutsche Forschungsgemeinschaft (DFG, German Research Foundation) – CRC 1456 (project-ID 432680300), projects B06 and C03. 
A.W. additionally acknowledges support by DFG RTG 2756 (project-ID 449750155), project A4.
M.B. acknowledges support from DESY (Hamburg, Germany), a member of the Helmholtz Association HGF and from the German
Research Foundation, project BU 2327/20-1.
A.H. acknowledges support from the Research Council of Finland with the Flagship of Advanced Mathematics for Sensing Imaging and Modelling proj. 359186. Centre of Excellence of Inverse Modelling and Imaging proj. 353093, and the Academy Research Fellow proj. 338408. 

\begin{figure}[H]
    \centering

    \includegraphics[width=0.7\textwidth]{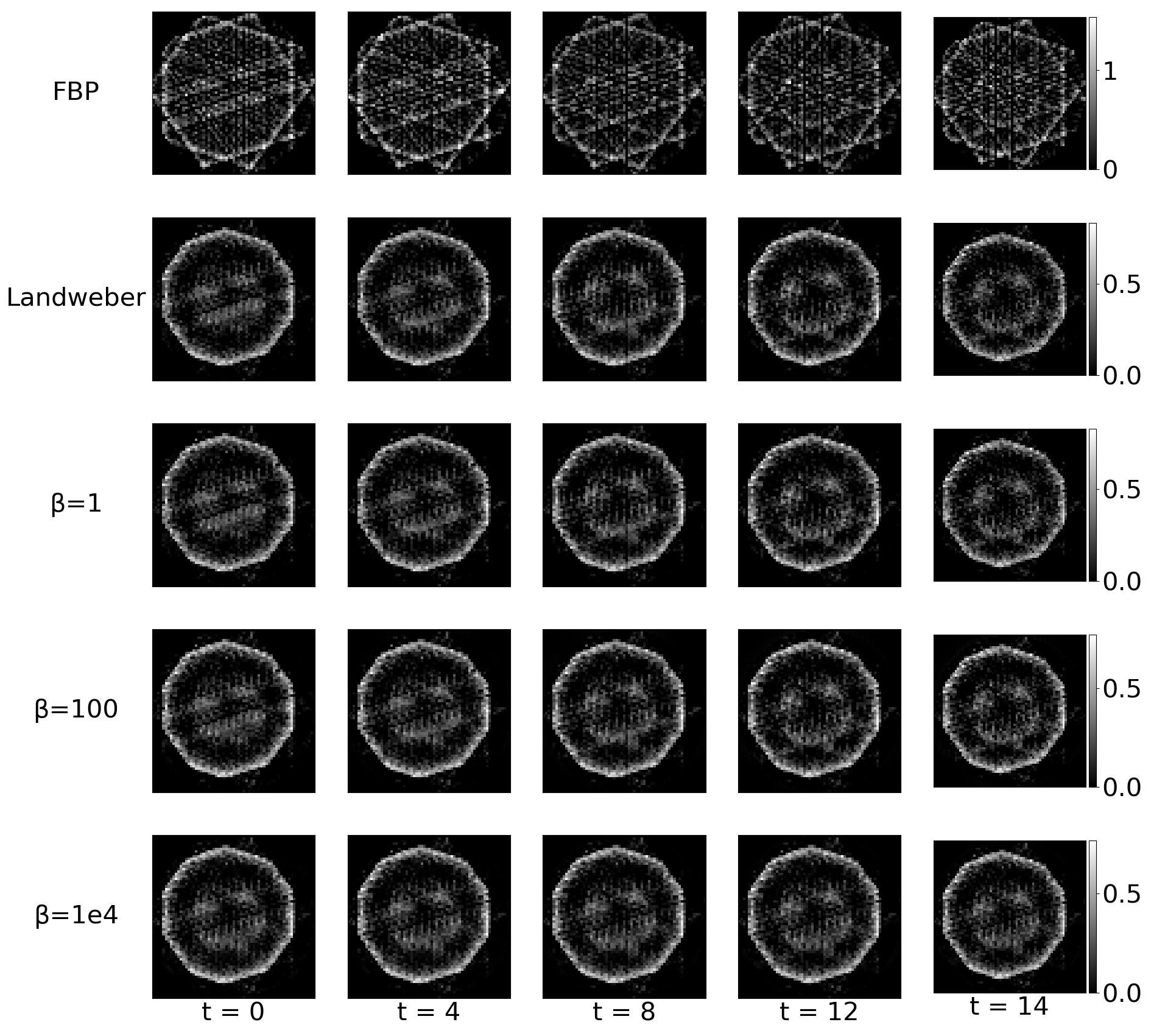}
    \caption{\textbf{Experiment description:} Comparison of FBP, Landweber and Algorithm~\ref{algo:timeDerivative} for $\gamma = 10$ and different regularization parameters $\beta$ for the emoji data
    at time points $t = 0,5,10,14$ for the same $5$ equidistantly spaced angles per time step. We added the non-negativity constraint for all methods.}
    \label{fig:CTrecoEmoji}
\end{figure}

\appendix
\vspace{-1ex}
\section{Proof for Lemma~\ref{lem:ThorstenKazimierski}}
    \textbf{Step 1:} $(A)$ is equivalent to
    \begin{enumerate}[label=(\alph*)]
        \item $\exists L'>0$ such that $\forall x,y \in X$ with $\| y \| \leq \| x \| = 1$, $\| x -y \| \leq \frac{1}{2}$ holds $ \| j_s(x) - j_s(y) \|_{X^*} \leq L' \| x - y \|_X^{s-1}.$
     %   \begin{align*}
      %      \| j_s(x) - j_s(y) \|_{X^*} \leq L' \| x - y \|_X^{s-1}.
      %  \end{align*}
    \end{enumerate}
        We %obviously 
        have $(A) \Rightarrow (a)$ with $L' = L$. 
        To prove $(a) \Rightarrow (A)$ it suffices to consider the case $\| x \| \geq \| y\|$ due to symmetry. Furthermore, we can express $x$ and $y$ by $ax'$ and $ay'$ for $\| x' \| = 1, a>0$. %Since $j_s(ax) = a^{s-1} j_s(x)$ (see 
        With Lemma~\ref{lem:dualityPtoQ}, we get
        \vspace{-1ex}
        \begin{align*}
            a^{s-1} \| j_s(x') - j_s(y') \| \leq L' a^{s-1} \| x' - y' \|^{s-1}
        \end{align*}
        and can consider without loss of generality $\| y \| \leq \| x \| = 1$.
        If $\| x - y \| > \frac{1}{2}$, we get
        \vspace{-1ex}
        \begin{align*}
            \| j_s(x) - j_s(y) \| \leq \underbrace{\| j_s(x) \|}_{=\|x\|^{s-1}} + \underbrace{\| j_s(y) \|}_{=\|y\|^{s-1}} \leq 1 + 1 = 2^s \left( \frac{1}{2} \right)^{s-1} < 2^s \| x -y \|^{s-1}.
        \end{align*}
      %  and thus $(A)$ holds with $L = \max(L', 2^s)$. \\
    \textbf{Step 2:} The proof is analogous to Step 1: $(B)$ is equivalent to 
    \begin{enumerate}[label=(\alph*)]
        \setcounter{enumi}{1}
        \item %For all 
        $\forall  x,y \in X$ with $\| y \| \leq \| x \| = 1$ $\| x -y \| \leq \frac{1}{2}$ we find $D'>0$ such that
           $ \| j_q(x) - j_q(y) \| \leq D' \| x - y \|^{s-1}$.
    \end{enumerate}
    %The proof is analogous to Step 1. 
\begin{comment} Since $ \max(\|x\|,\|y\|)^{q-s} = 1$ for $\| y \| \leq \| x \| = 1$ we have $(B) \Rightarrow (b)$ with $D' = D$. 
    To show $(b) \Rightarrow (B)$ we can consider without loss of generality $\| y \| \leq \| x \| = 1$ as in the first step since $\max(\|tx'\|, \|ty'\|)^{q-s} = t^{q-s} \max(\|x'\|,\|y'\|)^{q-s}$ and thus both sides of inequality $(B)$ are multiplied by $t^{q-1}$. 
    If $\| x -y \| > \frac{1}{2}$, we get
    \begin{align*}
        \| j_q(x) - j_q(y) \| < 2^s \left( \frac{1}{2} \right)^{s-1} < 2^s \| x -y \|^{s-1} = 2^s \underbrace{\max(\|x\|,\|y\|)^{q-s}}_{=1} \| x -y \|^{s-1},
    \end{align*}
    and thus $(B)$ holds with $D = \max(D', 2^s)$. \\ 
\end{comment}

    \textbf{Step 3:}
    We show $(a) \Leftrightarrow (b)$. Let $\| y \| \leq \| x \| = 1$ and $\| x -y \| \leq \frac{1}{2}$ and we start by showing $(a) \Rightarrow (b)$. 
    With Lemma~\ref{lem:dualityPtoQ} and the definition of the duality mapping we have
    \begin{align*}
        \| j_q(x) - j_q(y) \|
        &= \Bigl \| \underbrace{\| x \|^{q-s}}_{=1} j_s(x) - \| y \|^{q-s} j_s(y) + j_s(y) - j_s(y) \Bigr \| 
        \leq \| j_s(x) - j_s(y) \| + \| j_s(y) \| \Bigl | \| x \|^{q-s} - \| y \|^{q-s} \Bigr | \\
        &\leq L' \|x -y\|^{s-1} + \underbrace{\| y \|^{s-1}}_{\leq 1} K \| x -y \| 
        \leq L' \|x -y\|^{s-1} + K \| x -y \|^{2-s} \| x -y \|^{s-1} \\
        & \leq L' \| x - y \|^{s-1} + K \left( \frac{1}{2} \right)^{2-s} \| x -y \|^{s-1} 
         = \left(L' + K \left( \frac{1}{2} \right)^{2-s} \right) \| x -y \|^{s-1},
    \end{align*}
    which is inequality $(b)$ with $D' = L' + K 2^{s-2} $ where $K$ is the Lipschitz constant of the function $f: [\frac{1}{2} , 1] \to \R$ with $t \mapsto t^{q-s}$. This function can be used since $1 - \|y\| = \|x\| - \|y\| \leq \|x-y\| \leq \frac{1}{2}$ and thus $\frac{1}{2} \leq \|y\| \leq \|x\| = 1$. \\
    Now, we show $(b) \Rightarrow (a)$. Using again Lemma~\ref{lem:dualityPtoQ} we get as above
    \begin{align*}
        \| j_s(x) - j_s(y) \| 
        &= \| \|x\|^{s-q} j_q(x) - \|y\|^{s-q} j_q(y) \| 
        \leq \| j_q(x) - j_q(y) \| + \|j_q(y) \| \left | \|x\|^{s-q} - \|y \|^{s-q} \right | \\
        &\leq \left(D' + K' \left( \frac{1}{2} \right)^{2-s} \right) \| x -y \|^{s-1},
    \end{align*}
    which is inequality $(a)$ with $L' = D' + K' 2^{s-2} $ where $K'$ is the Lipschitz constant of the function $f: [\frac{1}{2} , 1] \to \R$ with $h \mapsto h^{s-q}$.

\bibliographystyle{plain}
\bibliography{litDIP}

\pagestyle{myheadings}
\thispagestyle{plain}

\end{document}